\documentclass[12pt,a4paper]{amsart}

\calclayout

\usepackage{xcolor}
\usepackage[lite]{amsrefs}
\usepackage{amssymb}
\usepackage[all,cmtip]{xy}

\usepackage{mathrsfs}
\usepackage{tabularx}
\usepackage{booktabs}
\usepackage[labelfont=bf,format=plain,justification=raggedright,singlelinecheck=false]{caption}
\usepackage{bookmark}
\usepackage{hyperref}
\usepackage{comment}

\newcommand{\F}{\mathbb{F}}

\newcommand{\X}{\mathcal{X}}

\newcommand{\Fq}{\mathbb{F}_q}
\newcommand{\Fqq}{\mathbb{F}_{q^2}}

\newcommand{\mP}{\mathbb{P}}

\newcommand{\D}{\mathcal{D}}
\newcommand{\eps}{\varepsilon}
\newcommand{\C}{\mathcal{C}}

\renewcommand{\div}{\operatorname{div}}

\numberwithin{equation}{section}

\theoremstyle{plain}
\newtheorem{theorem}[equation]{Theorem}
\newtheorem{corollary}[equation]{Corollary}
\newtheorem{lemma}[equation]{Lemma}
\newtheorem{proposition}[equation]{Proposition}

\theoremstyle{definition}

\theoremstyle{remark}
\newtheorem{remark}[equation]{Remark}

\begin{document}

\title[Uniqueness of a maximal function field with the third largest genus]{Uniqueness of the maximal function field with the third largest genus in characteristic $3$}
\author{Peter Beelen}
\address{Department of Applied Mathematics and Computer Science, Technical University of Denmark, DK-2800, Kongens Lyngby, Denmark}
\email{pabe@dtu.dk}

\author{Maria Montanucci}
\address{Department of Applied Mathematics and Computer Science, Technical University of Denmark, DK-2800, Kongens Lyngby, Denmark}
\email{marimo@dtu.dk}

\author{Jonathan Niemann}
\address{Department of Applied Mathematics and Computer Science, Technical University of Denmark, DK-2800, Kongens Lyngby, Denmark}
\email{jtni@dtu.dk}

\subjclass[2010]{} 

\begin{abstract}
Let $n$ be a positive integer and define $q=3^n$. It is well known that there exists a maximal function field $Z_3$ over $\Fqq$ of genus $g_3=q(q-1)/6$ and that $g_3$ is the third largest genus a maximal function field over $\Fqq$ can have. In this article we show that any maximal function field over $\Fqq$ of genus $g_3$ is isomorphic to $Z_3$. This extends previously known uniqueness result for the largest and second largest genera of maximal curves.
\end{abstract}

\maketitle

\section{Introduction}

An algebraic function field $F$ over a finite field with square cardinality $\mathbb{F}_{q^2}$ is called $\mathbb{F}_{q^2}$-maximal if it attains the Hasse--Weil upper bound with equality. More precisely, if $g(F)$ denotes the genus of $F$ then being $\mathbb{F}_{q^2}$-maximal means that
$$N(F) =q^2 +1+2qg(F),$$
where $N(F)$ denotes the number of places of degree one (we will refer to them also as $\mathbb{F}_{q^2}$-rational places)
of $F$. Maximal function fields, particularly those with large genus, have been and still are intensively
investigated in the literature, both for their theoretical interest (as extremal objects) and for their applications in coding theory. In fact, according to Goppa's construction \cite{Goppa}, algebraic function fields (and particularly those with many places of degree one) are most likely to give rise to good error-correction codes. 

Partial classification results for maximal function fields are known, however a full classification over a given finite field $\mathbb{F}_{q^2}$ is still an open problem. 

It is well known that the genus $g(F)$ satisfies $g(F) \leq q(q -1)/2 =: g_1$, see \cite{Ihara}, and that $g(F)=g_1$ if and only if $F$ is $\mathbb{F}_{q^2}$-isomorphic to the Hermitian function field 
$$X_1=\mathbb{F}_{q^2}(x,y), \quad y^q +y=x^{q+1},$$
see \cite{ruckCharacterizationHermitianFunction1994}.

Proceeding in analyzing the spectrum of large genera of maximal function fields, Fuhrmann, Garcia and Torres in \cite{fuhrmannMaximalCurves1997} proved that either $g(F) \leq \lfloor (q -1)^2/4\rfloor =: g_2$, or $g(F) = g_1$. For $q$
odd, $g(F)=g_2$ occurs if and only if $F$ is $\mathbb{F}_{q^2}$-isomorphic to the function field

$$X_2=\mathbb{F}_{q^2}(x,y), \quad  y^q +y =x^{\frac{q+1}{2}},$$
see \cite[Theorem 3.1]{fuhrmannMaximalCurves1997}. For $q$ even, a similar (but weaker) result is first obtained in \cite{AT}. Indeed 
the uniqueness of an $\mathbb{F}_{q^2}$-maximal function field $F$ ogf genus $g(F) = g_2$ is ensured under the extra 
condition that the function field $F$ has a particular Weierstrass place of degree one. If the extra condition is met, 
then it is proven in \cite{AT} that $g(F) = g_2 =q(q-2)/4$ 
if and only if $F$ is $\mathbb{F}_{q^2}$-isomorphic to the function field

$$Y_2=\mathbb{F}_{q^2}(x,y), \quad y^{\frac{q}{2}}+\ldots+y^2+x=x^{q+1}.$$

The full proof of uniqueness without the extra assumption on the particular Weierstrass place was then later obtained in \cite[Theorem 10.45]{hirschfeldAlgebraicCurvesFinite2008}. 

The value of the third largest genus of an $\mathbb{F}_{q^2}$-maximal function field $F$ has been computed in 
\cite{korchmarosGenusMaximalCurve2002}, where it is proven that either $g(F) \leq \lfloor (q^2 -q +4)/6\rfloor =: g_3$, or $g(F) = g_2$ or  $g(F) = g_1$. In \cite[Remark 3.4]{korchmarosGenusMaximalCurve2002} it is further shown that $g(F) \leq g_3$ is the best bound possible, as 
there exist $\mathbb{F}_{q^2}$-maximal function fields of genus $g_3$, namely

\begin{itemize}
    \item $X_3=\mathbb{F}_{q^2}(x,y), \quad x^{\frac{q+1}{3}} +x^{2\frac{q+1}{3}} +y^{q+1} =0, \ \textrm{if} \ q \equiv 2 \pmod 3;$
    \item $Y_3=\mathbb{F}_{q^2}(x,y), \quad y^q -yx^{{2(q-1)}{3}} +x^{\frac{q-1}{3}} =0, \ \textrm{if} \ q \equiv 1 \pmod 3;$
    \item $Z_3=\mathbb{F}_{q^2}(x,y), \quad y^q +y=\bigg(\sum_{i=1}^n x^{\frac{q}{p^i}}\bigg)^2 =0, \ \textrm{if} \ q =3^n.$
\end{itemize}

Understanding whether or not these function fields are the only $\mathbb{F}_{q^2}$-maximal function fields of genus $g_3$ is a well-known open problem. Researchers in the field of algebraic function fields tried to approach the problem from different perspectives. On one hand trying to classify algebraic function fields of Frobenius dimension three (which a function field of genus $g_3$ would have) \cite{FanaliGiulietti}, which gave uniqueness results for small values of $q$ (namely $q\leq 11$, see also \cite{FanaliGiuliettiPlatoni}), on the other hand studying Weierstrass places on $X_3$, $Y_3$ and $Z_3$ with the aim of finding a special Weierstrass place to use in a similar approach as in \cite{AT}, see \cite{BMV1,BMV2,BMV3}.  

In this paper, in the case $q=3^n$, combining St\"ohr--Voloch theory and embedding results in cubic surfaces (as a new tool that has never been used before in connection with maximal function fields), we manage to prove existence of a special Weierstrass place of degree one, which in additional combination with the famous Natural Embedding Theorem \cite{Korchmáros_Torres_2001}, allowed us to prove uniqueness of a maximal function field of genus $g_3$.

Hence the main result of this paper is the following.

\begin{theorem}\label{thm:main}
    Let $q=3^n\geq 9$, and suppose $F$ is an $\mathbb{F}_{q^2}$-maximal function field of genus 
    $$
        g(F) = \frac{q(q-1)}{6}.
    $$
    Then, there exist $x,y \in F$ such that $F = \Fqq(x,y)$ and
    $$
        y^q + y = \left(x + x^3 + \dots + x^{\frac{q}{3}}\right)^2,
    $$
    i.e., for $q = 3^n\geq 9$ there is just one $\Fqq$-maximal function field of genus $q(q-1)/6$ up to isomorphism, and this function field is $Z_3$.
\end{theorem}

Note that the case $q=3$ was already known in the literature, explaining our assumption that $q \geq 9$. Indeed, in this case $g(F)=1$ is in fact the second largest possible genus for an $\F_{3^2}$-maximal function field. Therefore, it follows from \cite[Theorem 3.1]{fuhrmannMaximalCurves1997} that $F$ is isomorphic to the function field $\F_{3^2} (x,y)$ with $y^3 +y = x^2$.

The paper is organized as follows. In Section 2 the necessary preliminary results on maximal function fields and curves (particularly of genus $g_3$), as well as some basic definitions of St\"or--Voloch theory are recalled. Proceeding then with the proof of Theorem \ref{thm:main}, we show in Section 3 how St\"or--Voloch theory can be used to find the special Weierstrass place of degree one. Lastly in Section 3 we complete the proof of Theorem \ref{thm:main} by providing a defining equation of a maximal function field of genus $g_3$ and showing how it can be simplified to the defining equation of the function field $Z_3$ (implying that they are isomorphic). 

\section{Preliminaries}\label{sec:preliminaries}
We collect some general results and introduce notation that will be useful later. We start by citing a result due to Weil on canonical divisors.

\begin{lemma}[{\cite[Chapter XIII, Section 12, Theorem 13, p. 291]{weilBasicNumberTheory1967
}}]\label{lemma:canonical_half_Weil}
    For any function field over a finite field, there exists a divisor $M$ such that $2M$ is a canonical divisor.
\end{lemma}

Now, let $\X$ be an $\mathbb{F}_{q^2}$-maximal curve of genus $g(\X)$, i.e., a projective, geometrically irreducible, non-singular algebraic curve over $\F_{q^2}$, with exactly $q^2+1+2g(\X)q$ rational points, and let $F/\Fqq$ be the function field of $\X$. Let $\overline{\F}_{q^2}$ be a fixed algebraic closure of $\Fqq$. When we apply St\"ohr--Voloch theory, in particular when considering order sequences and the ramification divisor, we will work in the constant field extension $\overline{\F}_{q^2}F/\overline{\F}_{q^2}$ of $F/\Fqq$. Throughout the paper we use both the language of curves and that of function fields. In particular, we will use the correspondence between points and places repeatedly. For example, the set of $\Fqq$-rational places of $F$ can be identified with $\X(\Fqq)$, the set of $\Fqq$-rational points of $\X$. With slight abuse of notation, we will also say that a place $P$ of $\overline{\F}_{q^2}F$ is $\Fqq$-rational or that $P \in \X(\Fqq)$, if $P \cap F$ is an $\Fqq$-rational place of $F$.

Define $\D := |D|=|(q+1)P_0|$, with $P_0 \in \X(\Fqq)$. Then $\D$ is the Frobenius linear series, and the Frobenius dimension of $\X$ is $r := \dim \D = \dim_{\Fqq}\left(\mathcal{L}((q+1)P_0)\right)-1$. Let $f_0, f_1, \dots, f_r$ be a basis of $\D$. By the Natural Embedding Theorem (see \cite{Korchmáros_Torres_2001}), there is an $M \leq r$ such that, after possibly making an $\Fqq$-linear change of basis,
$$
    \pi = (f_0 : \dots : f_M): \X \to \mP^M
$$
is a non-degenerate $\mathbb{F}_{q^2}$-birational morphism and $\pi(\X)$ is a curve of degree $q+1$ lying on a non-degenerate Hermitian variety. We assume without loss of generality, as in \cite{Korchmáros_Torres_2001}, that the Hermitian variety is in its canonical form, i.e., 
$$
    f_0^{q+1} + f_1^{q+1} + \cdots + f_M^{q+1} = 0.
$$
\begin{remark}
    We will mostly be concerned with the case where $M = r = 3$, but we stick with the general setting of \cite{Korchmáros_Torres_2001} for now. 
\end{remark}

Let $\eps_0, \eps_1, \dots, \eps_M$, denote the generic order sequence associated to $\pi$, and note that $\eps_0 = 0$, $\eps_1 = 1$, and $\eps_M = q$ follows from maximality of $\X$ (see \cite[page 105]{Korchmáros_Torres_2001}). Let $R$ denote the ramification divisor in the sense of \cite{stohrWeierstrassPointsCurves1986}, and define 
$$
    T := \sum_{P \in \X(\Fqq)} P \hspace{1cm} \text{ and } \hspace{1cm} E := R-T.
$$
Since the $M$'th order at an $\Fqq$-rational point is $q+1$ it follows from \cite[Theorem 1.5]{stohrWeierstrassPointsCurves1986} that $E$ is an effective divisor.

Fix a separating element $z\in F$ and let $D_z^{(i)}$ denote the $i$'th Hasse derivative with respect to $z$ (see, e.g., \cite[Section 5.10]{hirschfeldAlgebraicCurvesFinite2008} for properties of Hasse derivatives). We will need the function 
$$
    b_z := \sum_{i=0}^M f_i D_z^{(q)}(f_i^q) \in F
$$
on several occasions. First, to prove the following consequence of \cite[Lemma 4.7]{Korchmáros_Torres_2001}. We will use that standard notation $A \sim B$ to indicate that two divisors $A$ and $B$ are linearly equivalent. Further for a function $f$, its principal divisor will be denoted by $\div(f)$.

\begin{lemma}\label{lemma:divisor_T}
    With notation as above, it holds that 
    $$
        T = \div(b_z) + q \div(dz) + (q+1)D.
    $$
    In particular,
    $$
        T \sim qW + (q+1)D,
    $$
    where $W$ is a canonical divisor of $F$.
\end{lemma}

\begin{proof}
    Fix a point $P$ on $\X$ and a uniformizer $t$ at $P$. Write $h := t^{v_P(D)}$ such that all $hf_i$ are regular at $P$ (in fact this only changes the $f_i$'s when $P=P_0$). Now \cite[Lemma 4.7]{Korchmáros_Torres_2001} implies that
    $$
        v_P\left(\sum_{i=0}^M hf_i D_t^{(q)}\left((hf_i)^q\right)\right) = 
        \begin{cases}
            1 &\text{ if } P \in \X(\Fqq), \text{ and } \\
            0 &\text{ else.}
        \end{cases}
    $$

    On the other hand, by properties of Hasse derivatives and the fact that $\sum_{i=0}^Mf_i^{q+1} = 0$, we have
    \begin{align*}
        \sum_{i=0}^M hf_i D_t^{(q)}\left((hf_i)^q\right) &= \sum_{i=0}^M{hf_i\left(D_t^{(1)}(hf_i)\right)^q} \\
            &= \sum_{i=0}^M{h^{q+1}f_i\left(D_t^{(1)}(f_i)\right)^q} + h\left(D_t^{(1)}(h)\right)^q\sum_{i=0}^M{f_i^{q+1}} \\
            &= h^{q+1} \sum_{i=0}^M f_i\left( D_t^{(1)}(z) D_z^{(1)}(f_i)\right)^q \\
            &= h^{q+1} \cdot \left( dz/dt \right)^q \cdot b_z.
    \end{align*}

    Combining the above yields
    $$
        (q+1) \cdot v_P(D) + q\cdot v_P(dz) + v_P(b_z) =
        \begin{cases}
            1 &\text{ if } P \in \X(\Fqq), \text{ and } \\
            0 &\text{ else.}
        \end{cases}
    $$
    Since $P$ was arbitrary the result follows.
\end{proof}

Next we cite a known lemma, relating the divisor of $b_z$ to the divisor of the Wronskian $w := \det\left( D_z^{(\eps_i)}f_j\right)_{i,j = 0, \dots, M}$.

\begin{lemma}[{\cite[Lemma 4.4]{Korchmáros_Torres_2001}}]\label{lemma:divisor_of_b_z}
    With notation as above, it holds that
    $$
        q\div(f_M)+ \div(w) = \div(b_z)+ \div(\Delta),
    $$
    where
    $$
        \Delta := \det
        \begin{pmatrix}
            f_0 & f_1 & \cdots & f_{M-1} \\
            D_z^{(1)}f_0 & D_z^{(1)} f_1 & \cdots & D_z^{(1)} f_{M-1}\\
            \vdots & \vdots & \ddots & \vdots \\
            D_z^{(\eps_{M-1})}f_0 & D_z^{(\eps_{M-1})} f_1 & \cdots &D_z^{(\eps_{M-1})} f_{M-1} 
        \end{pmatrix}.
    $$
\end{lemma}

Combining the previous lemmas yields a useful description of $E = R-T$:

\begin{corollary}\label{cor:divisor_E}
    With notation as above, we have
    $$
        E = \div(\Delta) - q \div (f_M) + \left(\sum_{i=0}^M \eps_i - q \right)\div(dz) + (M-q) D
    $$
\end{corollary}

\begin{proof}
    In fact, by the definition of the ramification divisor (see \cite{stohrWeierstrassPointsCurves1986}) and the results above, we have
    \begin{align*}
        E = R-T &= \left(\div(w) + (\eps_1 + \dots + \eps_M)\div(dz) + (M+1)D\right) \\
           & \quad - \left( \div(b_z) + q \div(dz) + (q+1)D\right) \\
           &= \div(\Delta) - q \div (f_M) + \left(\sum_{i=0}^M \eps_i - q \right)\div(dz) + (M-q) D,
    \end{align*}
    as claimed.
\end{proof}

From now on we suppose that $q=3^n \geq 9$ and that $\X$ is $\Fqq$-maximal of genus
$$
    g(\X) = g_3 = \frac{q(q-1)}{6},
$$
so that $\X$ has the third largest possible genus for a maximal curve over $\mathbb{F}_{q^2}$. The following proposition collects some well-known facts about such a curve.

\begin{proposition}\label{prop:frob_generic_order_sequence}
    The Frobenius dimension of $\X$ is $r=3$, and the embedding associated with $\mathcal{D} = |(q+1)P_0|$ maps $\X$ to a smooth curve of degree $q+1$ lying on a non-degenerate Hermitian surface in $\mathbb{P}^3$. Moreover, the generic order sequence of this embedding is $(0,1,2,q)$.
\end{proposition}
\begin{proof}
    The fact that $r=M=3$ and that $\mathcal{D}$ maps $\X$ to a smooth curve of degree $q+1$ lying on a non-degenerate Hermitian surface in $\mathbb{P}^3$ are direct consequences of Castelnuovo's bound \cite[Corollary 10.25]{hirschfeldAlgebraicCurvesFinite2008} and the Natural Embedding theorem \cite{Korchmáros_Torres_2001}. The generic order sequence $(0,1,2,q)$ is obtained from \cite[Remark 2.7]{oliveiraSpaceMaximalCurves2019}.
\end{proof}

We fix a basis $(f_0, f_1, f_2, f_3)$ of $\mathcal{L}(D)$ such that 
$$
    \pi = (f_0:f_1:f_2:f_3): \X \to \mathbb{P}^3
$$
is the embedding described above, and assume that the Hermitian surface is in its canonical form, i.e., that
$$
    f_0^{q+1} + f_1^{q+1} + f_2^{q+1} + f_3^{q+1} = 0.    
$$

In \cite[Question 2.10]{oliveiraSpaceMaximalCurves2019} the authors ask whether characteristic zero generalizations of Castelnuovo's and Halphen's results also hold in positive characteristic. A version of this question was answered by Macrì and Schmidt in \cite[Theorem 1.1 and p. 154]{Macri_Schmidt_2020}, and their results imply that $\pi(\X)$ is contained in a cubic:

\begin{proposition}\label{prop:contained_in_cubic}
    The curve $\pi(\X) \subseteq \mP^3$ is contained in a unique absolutely irreducible cubic surface defined over $\Fqq$. 
\end{proposition}

\begin{proof}
    It follows from \cite[Theorem 1]{korchmarosGenusMaximalCurve2002} that $\pi(\X)$ does not lie on a quadric, and linear independence of the $f_i$'s guarantees that it does not lie on a hyperplane. For $q \geq 27$ it follows from \cite[Theorem 1.1]{Macri_Schmidt_2020} that $\pi(\X)$ lies on a cubic since otherwise we would have $g(\X) \leq \frac{(q+1)^2}{8}+1$, which is false. For $q=9$, the Riemann--Roch Theorem yields
    $$
        \dim \mathcal{L}(3D) = 30 + 1 - 12 = 19,
    $$
    since $\deg (3D) = 30 > 22 =  2g_3-2$, while $\Fqq[X_0,X_1,X_2,X_3]_3$, the $\Fqq$-space of degree three homogeneous polynomials in four variables, has dimension $\binom{6}{3} = 20$. Hence, the $\Fqq$-linear map 
    $$
        \mu : 
        \begin{cases}
            \Fqq[X_0,X_1,X_2,X_3]_3 &\longrightarrow \hfil \mathcal{L}(3D) \\
            \hfil G  &\longmapsto \ G(f_0,f_1,f_2,f_3)
        \end{cases}
    $$
    has a non-trivial kernel, which gives rise to a cubic containing $\pi(\X)$ also for $q=9$.

    In either case, we claim that the cubic is absolutely irreducible, since $\pi(\X)$ would otherwise be contained in a quadric. Further $\pi(\X)$ cannot be contained in two distinct cubic surfaces, since their intersection would have degree at most $9$, which is strictly less than $\deg \pi(\X) = q+1$. This uniqueness result implies that the cubic is defined over $\Fqq$, since acting with the Frobenius would otherwise give a different cubic surface also containing $\pi(\X)$.
\end{proof}

\section{St\"ohr--Voloch theory and a special rational point}

We keep the notation and assumptions from the previous section, i.e., $q = 3^n \geq 9$ and $\X$ is $\Fqq$-maximal of genus
$$
    g(\X) = g_3 = \frac{q(q-1)}{6}.
$$
The Hermitian and cubic equations guaranteed by Propositions \ref{prop:frob_generic_order_sequence} and \ref{prop:contained_in_cubic} will be the starting point for showing that $\X$ is birationally equivalent to the known third largest genus curve. However, we need to know more about the embedding $\pi = (f_0:f_1:f_2:f_3)$ first. In the following, we will use St\"ohr--Voloch theory to study this embedding. \newline

\subsection{The ramification divisor}

Recall that the ramification divisor with respect to $\pi$ in the sense of \cite{stohrWeierstrassPointsCurves1986} is given by 
\begin{align*}
    R := \div (w) + \left(\sum_{i=0}^3 \eps_i\right) \div(dz) + (3+1)D = \div(w) + (q+3) W + 4D,
\end{align*}
where
$$
    w = \det\left( D_z^{(\varepsilon_i)}f_j\right)_{i,j = 0, \dots, 3},
$$ 
$z$ is a separating variable, $W = \div(dz)$ and $D = (q+1)P_0$ for a fixed $\Fqq$-rational point $P_0$. 
Moreover, from Corollary \ref{cor:divisor_E} we have
$$E=\div(\Delta)-q\div(f_3)+3\div(dz)-(q-3)D.$$
As mentioned in Section \ref{sec:preliminaries}, $v_P(R) \geq 1$ for all $\Fqq$-rational $P$, so 
$$
    E = R - T \geq 0,
$$
where $T$ is the sum of all $\Fqq$-rational points. The first important result in this section is that $E$ is three times another effective divisor. 

\begin{lemma}\label{lemma:divisor_A}
    There is an effective divisor $A$ of degree $\frac{q}{3}-1$ such that 
    $$
        3A = E.
    $$
    Moreover, $\ell(A)=1$ so $A$ is the unique effective divisor in its divisor class, and $A$ satisfies
    $$
        A \sim W - \left(\frac{q}{3}-1\right)D.
    $$
\end{lemma}

\begin{proof}
    The existence of a divisor $A$ with $3A = E$ follows from Corollary \ref{cor:divisor_E} if we can show that $\Delta$ is a cube, where 
    $$
        \Delta = \det
        \begin{pmatrix}
            f_0 & f_1 & f_{2} \\
            D_z^{(1)}f_0 & D_z^{(1)} f_1 & D_z^{(1)} f_{2}\\
            D_z^{(2)}f_0 & D_z^{(2)} f_1 & D_z^{(2)} f_{2} 
        \end{pmatrix}.
    $$
    Since the characteristic of $F$ is $3$, $\Delta$ is a cube if and only if $d \Delta = 0$ (see \cite[Proposition 3.10.2 (d) and Proposition 4.1.8 (c)]{stichtenothAlgebraicFunctionFields2009}). We have $d \Delta = (D_z^{(1)}\Delta) dz$, so $\Delta$ is a cube if and only if $D_z^{(1)} \Delta = 0$. Using standard properties of Hasse derivatives, we see that 
    \begin{align*}
        D_z^{(1)} \Delta &= \det
        \begin{pmatrix}
            D_z^{(1)}f_0 & D_z^{(1)} f_1 & D_z^{(1)} f_{2} \\
            D_z^{(1)}f_0 & D_z^{(1)} f_1 & D_z^{(1)} f_{2}\\
            D_z^{(2)}f_0 & D_z^{(2)} f_1 & D_z^{(2)} f_{2} 
        \end{pmatrix} 
        + \det
        \begin{pmatrix}
            f_0 & f_1 & f_{2} \\
            2D_z^{(2)}f_0 & 2D_z^{(2)} f_1 & 2D_z^{(2)} f_{2}\\
            D_z^{(2)}f_0 & D_z^{(2)} f_1 & D_z^{(2)} f_{2} 
        \end{pmatrix} \\
        & \quad+ \det
        \begin{pmatrix}
            f_0 & f_1 & f_{2} \\
            D_z^{(1)}f_0 & D_z^{(1)} f_1 & D_z^{(1)} f_{2}\\
            3D_z^{(3)}f_0 & 3D_z^{(3)} f_1 & 3D_z^{(3)} f_{2} 
        \end{pmatrix}. 
    \end{align*}

    The first two determinants are zero since there are linearly dependent rows and the third determinant is zero since the third row is zero. We conclude that $D_z^{(1)} \Delta = 0$, so $\Delta \in F^3$, i.e., there exists some $\delta \in F$ such that 
    $$
        \Delta = \delta^3.
    $$
    Now Corollary \ref{cor:divisor_E} implies that
    $$
        E = 3 \left( \div (\delta) - \frac{q}{3} \div (f_3) + \div(dz) - \left(\frac{q}{3}-1\right)D\right),
    $$
    so
    $$
        A := \div (\delta) - \frac{q}{3} \div (f_3) + \div(dz) - \left(\frac{q}{3}-1\right)D
    $$
    works. It follows immediately that $\deg A = (2g-2) - (q/3 - 1)(q+1) = q/3 - 1$, and 
    $$
        A \sim W - \left(\frac{q}{3}-1\right) D.
    $$
    Finally, we prove that $\ell(A) = 1$. If $\ell(A)>1$ then $\mathcal{L}(A)$ contains a non-constant function $f$. Let $d$ be the degree of the pole divisor of $f$. Then, since $A$ is effective, $d\leq \deg A $ and $d = [F : \Fqq(f)]$. Each $\Fqq$-rational place of $F$ lies above one of the $q^2+1$ rational places of $\Fqq(f)$, and there are at most $d$ rational places above each one. Using maximality of $F$ we get 
    $$
        \frac{q^3+2q^2+3}{3} = q^2 + 1 + 2g(F)q \leq d(q^2+1) \leq \deg(A) (q^2+1) = \frac{q^3 - q(3q-1)-3}{3},
    $$
    which is a contradiction. 
\end{proof}

\subsection{Orders at non-rational points}

The goal of this section is to show that all non-$\Fqq$-rational points have $\mathcal{D}$-order sequence $(0,1,2,q)$. A direct consequence of this is that $v_P(R) = 0$ for all non-$\Fqq$-rational $P$, i.e., $R$ and hence $E$ and $A$ are supported on the $\Fqq$-rational points of $\X$.

For the rest of this section, let $P$ be a non-$\Fqq$-rational point of $\X$. By \cite[Proposition 10.6]{hirschfeldAlgebraicCurvesFinite2008} the order sequence is of the form $(0,1,j_2,q)$, and our claim is that $j_2 = 2$. The sequence of non-gaps at $P$, i.e., elements of the Weierstrass semigroup $H(P)$, begins with
$$
    0 < m_1(P) < m_2(P) \leq q < m_3(P)
$$
(see \cite[Equation 10.9]{hirschfeldAlgebraicCurvesFinite2008}). More precisely, by \cite[Proposition 10.9 (XIV)]{hirschfeldAlgebraicCurvesFinite2008} we have $m_2(P) = q$ and \cite[Proposition 10.6 (XI)]{hirschfeldAlgebraicCurvesFinite2008} then implies
$$
    m_1(P) \in \{q-1,q-j_2\},
$$
since $q-m_1(P)$ is an order. We will show that neither option can occur if $j_2 > 2$. The key is the following lemma which gives us a way to construct a number of different gaps depending on the contribution of $P$ to $A$.

\begin{lemma}\label{lemma:gaps_at_non_rational}
    Let $P$ be a non-$\Fqq$-rational point with $j :=j_2(P) > 2$, and define 
    $$
        \Sigma_j := \left\lbrace e_1 + e_2\cdot j + e_3\cdot q \mid e_1 + e_2 + e_3 \leq \left(\frac{q}{3}-1\right), e_i\geq0 \right\rbrace.
    $$
    Then,
    $$
        v_P(A) + t + 1 \ \text{ is a gap at } \ P 
    $$
    for all $t \in \Sigma_j$.
\end{lemma}

\begin{proof}
    Suppose $P$ is a non-$\Fqq$-rational point with $j:= j_2(P) > 2$. The order sequence at $P$ is $(0,1,j,q)$, so we can find $\overline{\F}_{q^2}$-linear combinations $g_0 = 1,g_1,g_2,g_3$ of $f_0,f_1,f_2,f_3 \in \mathcal L(D)$ with 
    $$
        v_P(g_1)=1, \quad v_P(g_2) = j \quad \text{and} \quad v_P(g_3) = q.
    $$
    Let $t$ be any element of $\Sigma_j$ and let $e_1,e_2,e_3$ be nonnegative integers satisfying $e_1 + e_2 + e_3 \leq \left(\frac{q}{3}-1\right)$ with $t = e_1 + e_2\cdot j + e_3\cdot q$. We consider the function 
    $$
        f := g_1^{e_1}g_2^{e_2}g_3^{e_3} \in \mathcal{L}\left( \left(\frac{q}{3}-1\right)D\right).
    $$
    By Lemma \ref{lemma:divisor_A} we have that
    $$
        A+ \left(\frac{q}{3}-1\right)D + (f)
    $$
    is an effective canonical divisor. Since the valuation at $P$ of this divisor is exactly $v_P(A) + t$, it follows from \cite[Corollary 14.2.5]{villasalvadorTopicsTheoryAlgebraic2006} that $v_P(A) + t + 1$ is a gap, as claimed.
\end{proof}

We now prove the main result of this section. 

\begin{theorem}
    All non-$\Fqq$-rational points of $\X$ have order sequence $(0,1,2,q)$. In particular, $\operatorname{supp} A \subseteq \X(\Fqq)$.
\end{theorem}
\begin{proof}
    Suppose a non-$\Fqq$-rational point $P$ has order sequence $(0,1,j,q)$ with $j>2$. By \cite[Theorem 1.5]{stohrWeierstrassPointsCurves1986} this means $P$ is in the support of $A$. \newline

    \emph{Claim:} $j \in \left\{3,\frac{q}{3},\frac{q-1}{2},\frac{q+1}{2} \right\}$ \newline

    First note that $\deg(R-T)= \deg(E) = 3\deg(A) = q-3$, and $v_P(E) \geq j-2$ by \cite[Theorem 1.5]{stohrWeierstrassPointsCurves1986}. Moreover, acting with the Frobenius on $P$ gives at least one other non-$\Fqq$-rational point, say $Q$, for which it also holds that  $v_Q(E) \geq j-2$. We conclude from this that $2(j-2)\leq q-3$, i.e.,
    $$
        j\leq \frac{q+1}{2}.
    $$
    As in the proof of Lemma \ref{lemma:gaps_at_non_rational}, we choose a basis $(g_0 =1,g_1,g_2,g_3)$ of $\mathcal{L}(D)$ with $v_P(g_1)=1$, $v_P(g_2)=j$ and $v_P(g_3)=q$. Since $\pi(\X)$ is contained in a cubic by Proposition \ref{prop:contained_in_cubic} there is a cubic relation among the $f_i$'s, and writing each $f_i$ in terms of the new basis gives a cubic relation among the $g_i$'s. A term in this relation is of the form $g_1^ag_2^bg_3^c$ with $0\leq a+b+c\leq 3$, and such a term has valuation $a + bj + cq$ at $P$. The smallest occurring valuation must cancel, i.e., there must be at least two terms with this valuation. 

    Suppose $a + bj + cq = a' + b'j + c'q$ for some $a,b,c,a',b',c' \geq 0$ with $a+b+c \leq 3$ and $a'+b'+c' \leq 3$ and assume that $(a,b,c) \neq (a',b',c')$. If $c = c'$ then 
    $$
        j|b'-b| = |a-a'| \leq 3,
    $$
    which is a contradiction except for $j=3$, since $b'=b$ would imply also $a'=a$.

    If $|c-c'|\geq 2$, that either $|b-b'|<3$ or $|a-a'|<3$. Indeed, if $|b-b'|=3$ and $|a-a'|=3$, then necessarily $c=c'=0$. Hence we obtain 
    $$
        2q \leq q |c-c'| = |(a'+b'j)-(a+bj)|\leq \frac{3}{2}(q+1)+2,
    $$
    which implies $q \le 7$, a contradiction. Hence, we may assume without loss of generality that $c = c'+1$, so that 
    $$
        a + bj + q = a' + b'j.
    $$
    If $b'\leq 1$ then $q \le a' + b'j \le  2+\frac{q+1}{2}$, a contradiction. If $b' = 2$ then $b=2$ and $b=1$ give contradictions quickly, but with $b=0$ we obtain the valid options $j=\frac{q+1}{2}$ and $j= \frac{q-1}{2}$. Similarly, if $b'=3$ then $a'=0$ and we find that the only valid option is $j=\frac{q}{3}$. The claim follows. \newline

    Next, we exclude the options $j=3$, $j=q/3$ and $j=(q+1)/2$ one at a time using Lemma \ref{lemma:gaps_at_non_rational} and $m_1(P)\in \{q-1,q-j\}$. Suppose first that $j=3$. By Lemma \ref{lemma:gaps_at_non_rational} $v_P(A)+t+1$ is a gap at $P$ for all 
    \begin{align*}
        t &\in \Sigma_3 =\left\lbrace e_1 + e_2\cdot 3 + e_3 \cdot q \mid e_1 + e_2 + e_3 \leq \left(\frac{q}{3}-1\right)\right\rbrace.
    \end{align*}
    This set contains $0,1,\dots, q-5$ and $q-3$. In fact, $q-3 = 3(q/3 - 1)$ shows $q-3\in \Sigma_3$ and for $n \leq q-5$ we claim that $n = e_1 + 3 e_2$ with $e_1 \leq 2$ is valid. To see this, note that 
    $$
        3(e_1 + e_2) = n + 2e_1 \leq q-1.
    $$
    Since $e_1 + e_2$ is an integer this implies $e_1 + e_2 \leq q/3 - 1$, just as we wanted.

    If $v_P(A) = 1$ then $t=q-5$ and $t=q-3$ shows that $q-3$ and $q-1$ are gaps, which contradicts $m_1(P) \in \{q-j,q-1\}$. If $v_P(A)=2$ then $t=q-3$ shows that $q$ is a gap at $P$, which is impossible since $m_2(P)=q$. Finally, if $v_P(A) \geq 3$ then
    $$
        0 \leq (q-3) - (v_P(A)+1) < (q-1) - (v_P(A)+1) \leq q-5,
    $$
    where the first inequality follows from $v_P(A) \leq \deg A = q/3 - 1$, so $t = (q-3) - (v_P(A)+1)$ and $t= (q-1) - (v_P(A)+1)$ are valid choices, and we get that $q-j=q-3$ and $q-1$ are gaps, a contradiction. \newline

    The case $j=q/3$ can be excluded in a similar way. In fact, note that $v_P(A) <q/3-1$ as if $P$ is in the support of $A$ so is the $\mathbb{F}_{q^2}$-Frobenius image of $P$. This implies that $\frac{q}{3} - v_P(A) - 2 \geq 0$. Hence 
    $$
        (q-1) - (v_P(A)+1) = \left(\frac{q}{3} - v_P(A) - 2\right) + 2 \frac{q}{3} \in \Sigma_{q/3}
    $$
    since
    $$
        \left(\frac{q}{3} - v_P(A) - 2\right) + 2 \leq \frac{q}{3}-1,
    $$
    so $q-1$ is a gap, and 
    $$
        \left(q-\frac{q}{3}\right) - (v_P(A)+1) = \left(\frac{q}{3} - v_P(A) - 1\right) +  \frac{q}{3} \in \Sigma_{q/3}
    $$
    since
    $$
        \left(\frac{q}{3} - v_P(A) - 1\right) + 1 \leq \frac{q}{3}-1,
    $$
    so $q-j$ is also a gap.

    To exclude the case $j=(q+1)/2$, note first that $v_P(A)=\frac{q-3}{6}$ follows from the considerations in the beginning of this proof. Now
    $$
        (q-1)- \left(\frac{q-3}{6}+1\right) =  \left(\frac{q}{3}-2\right) + j\in \Sigma_j
    $$
    and 
    $$
        (q-j) - \left(\frac{q-3}{6}+1\right) = \frac{q}{3}-1 \in \Sigma_j,
    $$
    so both $q-1$ and $q-j$ are gaps. 

    To exclude the case $j=(q-1)/2$ a different strategy is needed, and we consider $q=9$ separately. Suppose first that $q\geq 27$, and that $j=(q-1)/2$. By the Riemann--Roch Theorem and Lemma \ref{lemma:divisor_A} we have 
    $$
        \ell\left(\left( \frac{q}{3} -1\right) D\right) = \left(\frac{q}{3}-1\right)(q+1) +2-g = \frac{q^2 -3q + 6}{6}, 
    $$
    but we will use $j=(q-1)/2$ to construct $ \frac{q^2 -3q + 6}{6}+1$ linearly independent elements in this Riemann--Roch space. 

    With $g_0=1, g_1, g_2, g_3$ as above the elements of
    $$
        B := \left\{g_1^ag_2^bg_3^c \ \big\vert \ a+b+c\leq \frac{q}{3}-1 \quad \text{and} \quad a=0 \text{ or } b\leq 1\right\}
    $$
    are in $\mathcal{L}((q/3-1)D)$, and we claim that they all have distinct valuations at $P$. In fact, if $(a,b,c)$ and $(a',b',c')$ are distinct valid triplets giving rise to the same valuation then 
    $$
        (a+c) + (b + 2c)j = a + bj+cq = a' + b'j + c'q = (a'+c') + (b'+2c')j,
    $$
    and since $a+c \le a+b+c \le q/3-1< (q-1)/2=j$ working modulo $j$ implies 
    $$
        a+c = a'+ c' \quad \text{and} \quad b + 2c = b'+2c'.
    $$
    If $c = c'$ we get $a=a'$ and $b=b'$ immediately, i.e., a contradiction. Otherwise we may assume $c>c'$, and then 
    $$
        a' = a + c-c' \geq 1 \quad \text{and} \quad b'=b+ 2(c-c')\geq 2,
    $$
    which is also a contradiction. 
    
    It follows that the elements of $B$ are $\Fqq$-linearly independent in $F$, and since
    $$
        |B| = \binom{\frac{q}{3}+2}{3} - \binom{\frac{q}{3}-1}{3} =  \frac{q^2-3q+6}{6},
    $$
    they form a basis of $\mathcal{L}((q/3-1)D)$. 

    Now, since $q\in H(P)$ we can pick a function $f \in F$ with $(f)_\infty=qP$, and then $fg_2^4 \in \mathcal{L}((q/3-1)D)$ since
    $$
        4(q+1) <\left(\frac{q}{3}-1\right)(q+1)
    $$
    using $q\geq 27$, and
    $$
        v_P(fg_2^4) = 4j-q = q-2. 
    $$
    There is no element in $B$ with valuation $q-2$ at $P$. In fact, we would need $c=0$ and $b\leq 1$, but then $a$ would have to be too large. We conclude that the functions in $B \cup \{fg_2^4\}$ are linearly independent, which is a contradiction. \newline
    
    Finally, suppose that $q=9$. Then the considerations in the beginning of the proof imply that $A = P + P'$, where $P'$ is the image of $P$ under the $\Fqq$-Frobenius. By Lemma \ref{lemma:divisor_A} we have
    $$
        P+P'=A \sim W - 2D,
    $$
    where $W$ is a canonical divisor. By Lemma \ref{lemma:canonical_half_Weil} we can write $W \sim 2M$ for some divisor $M$, and then \cite[Lemma 1]{ruckCharacterizationHermitianFunction1994} implies 
    $$
        5(P + P' - 2P_0) \sim 5(W-2D-2P_0) = 10(M-D-P_0) \sim 0.
    $$
    From the Fundamental Equivalence (see \cite[Corollary 1.2]{fuhrmannMaximalCurves1997}) we have $qP + P' \sim D$, which when combined with the above gives
    $$
       4P' - 4P \sim 0.
    $$
    This implies $4 \in H(P)$ which is a contradiction since $m_1(P) \in \{5,8\}$.

    We have reached a contradiction in all cases, so we conclude that $j=2$ is the only option for a non-$\Fqq$-rational point of $\X$. 
\end{proof}

\begin{corollary}\label{cor:at_least_one_rational_in_A}
    The divisor $A$ is supported on the $\Fqq$-rational points of $\X$. In particular, the support contains at least one $\Fqq$-rational point.
\end{corollary}
\begin{proof}
    This is a direct consequence of the above theorem, \cite[Theorem 1.5]{stohrWeierstrassPointsCurves1986} and the fact that $\deg A = \frac{q}{3}-1 \geq2$.
\end{proof}

    

\subsection{The special rational point}

The goal of this section is to show that the support of the divisor $A = \frac{R-T}{3}$ is exactly one point. We saw in the previous section that any point in the support of $A$ is $\Fqq$-rational, so we proceed by studying the order sequence at a rational point.

\begin{proposition}\label{prop:j_2_rational}
    For any $\Fqq$-rational point $P$ of $\X$, the $\mathcal{D}$-order sequence is $(0,1,j_2,q+1)$ with
    $$
        j_2 \in \{2,q/3 + 1\}
    $$
\end{proposition}

\begin{proof}
    Since $\pi(\X)$ is contained in a cubic it follows from \cite[Theorem 2.5]{oliveiraSpaceMaximalCurves2019}, and the fact that $q$ and $q+1$ are non-gaps at any $\Fqq$-rational point of an $\Fqq$-maximal curve (see, e.g., \cite[Proposition 10.6 (IV) and (XII)]{hirschfeldAlgebraicCurvesFinite2008}), that the order sequence at a $\Fqq$-rational point is $(0,1,j_2,q+1)$ with $j_2 \in \{2,3,q/3 + 1\}$. The only thing that remains to be shown is that $j_2 = 3$ does not occur.

    Suppose now that $P$ is a $\Fqq$-rational point of $\X$ with $j_2 = 3$. By \cite[Proposition 10.6]{hirschfeldAlgebraicCurvesFinite2008} the Weierstrass semigroup $H(P)$ contains $\langle q-2,q,q+1\rangle$. By \cite[Section A2]{Fuhrmann1995} (see also \cite[Lemma 3.4]{cossidenteCurvesCoveredHermitian1999}) the genus of the semigroup $\langle q-2,q,q+1\rangle$ is equal to $g(\X)$, hence
    $$
        H(P) = \langle q-2, q, q+1\rangle.
    $$
    We will arrive at a contradiction by showing that $n:=\frac{(q+1)(q-6)}{6}$ is both a gap and a non-gap at $P$. 

    First we show that $n$ is a gap. 
    In fact, 
    $$
        \frac{q-3}{6}\cdot (q-2) = \frac{q^2-5q+6}{6} > \frac{q^2-5q-6}{6} = \frac{(q+1)(q-6)}{6},
    $$
    so $(q-3)/6$ times the smallest generator is larger than $n$. On the other hand, 
    $$
        \left(\frac{q-3}{6}-1 \right)\cdot (q+1) = \frac{q^2-8q-9}{6} < \frac{q^2-5q-6}{6} = \frac{(q+1)(q-6)}{6},
    $$
    so $(q-3)/6 - 1$ times the largest is smaller than $n$. But then $n$ cannot be in the semigroup $\langle q-2, q, q+1\rangle$.

    Next, we construct a function with pole divisor equal to $\frac{(q+1)(q-6)}{6}P$, and hence get a contradiction. First of all, we claim that $m := \frac{q^2-2q+3}{3}$ is a gap of the semigroup $H(P)=\langle q-2, q, q+1\rangle$. Suppose the claim is false, then there exist non-negative integers $a,b$ and $c$ such that 
    \begin{equation}\label{eq:abc}
        a(q-2)+bq+c(q+1)=m.
    \end{equation} 
    Since $q/3 (q+1)>m$ and $q/3(q-2)<m<(q/3+1)(q-2)$, we conclude that $a+b+c=q/3$. But then Equation \eqref{eq:abc} implies that $2b+3c=1$. Since $b$ and $c$ are non-negative integers, this gives a contradiction. Hence $m\not\in H(P)$. 
    
    This implies using \cite[Corollary 14.2.5]{villasalvadorTopicsTheoryAlgebraic2006}, that there exists a regular differential $\omega$ of $F$ such that
    $$
        (\omega) = (m-1)P + E_P,
    $$
    where $E_P$ is an effective divisor of degree $q/3 - 2$ with support disjoint from $P$.

    Let $M$ be the divisor guaranteed by Lemma \ref{lemma:canonical_half_Weil}. Then, 
    $$
        (\omega) \sim 2M,
    $$
    and therefore
    $$
        \deg (M-(g-1)P) = 0.
    $$
    Using maximality of $F$, we may conclude that $(q+1)(M-(g-1)P)$ is a principal divisor, see \cite[Lemma 1]{ruckCharacterizationHermitianFunction1994}.
    This means that 
    
    \begin{align*}
        0 & \sim (q+1)\cdot (M-(g-1)P) \\
            &\sim \frac{q+1}{2} \cdot \left((\omega) - (2g-2)P\right) \\
            &= \frac{q+1}{2} \cdot \left((m-1)P + E_P - (2g-2)P\right) \\
            &= \frac{q+1}{2} \cdot \left(E_P - \left(\frac{q}{3}-2\right) P\right).
    \end{align*}
   Since $E_P$ is effective by construction, the above therefore shows that
    $$
        \frac{q+1}{2} \cdot \left(\frac{q}{3}-2\right) = \frac{(q+1)(q-6)}{6}
    $$
    is a non-gap at $P$. We have reached a contradiction, and conclude that $j_2=3$ never occurs for a $\Fqq$-rational point $P$ of $\X$. The result follows.
\end{proof}

\begin{corollary}\label{cor:special_place}
    There is an $\Fqq$-rational point $P$ such that
    $$
        A = \left( \frac{q}{3}-1\right)P.
    $$
    The order sequence at $P$ is $(0,1,q/3+1,q+1)$ and the Weierstrass semigroup is $H(P) = \langle 2q/3,q,q+1\rangle$.

\end{corollary}

\begin{proof}
    By Corollary \ref{cor:at_least_one_rational_in_A} there is at least one $\Fqq$-rational point, say $P$, in the support of $A$. Proposition \ref{prop:j_2_rational} above implies that the order sequence at $P$ is $(0,1,j_2,q+1)$ with $j_2 \in \{2,q/3 +1\}$. If $j_2 = 2$ then \cite[Theorem 1.5]{stohrWeierstrassPointsCurves1986} would imply $v_P(R)=1$ and hence $v_P(A) = \frac{1}{3}v_P(R-T) = 0$, which is a contradiction. We conclude that the order sequence is $(0,1,q/3+1,q+1)$. 

    It now follows from \cite[Proposition 10.6 (IV)]{hirschfeldAlgebraicCurvesFinite2008} that $H(P)\supseteq \langle 2q/3, q, q+1\rangle$. By \cite[Section A2]{Fuhrmann1995} (see also \cite[Lemma 3.4]{cossidenteCurvesCoveredHermitian1999}) the semigroup $ \langle 2q/3, q, q+1\rangle$ has genus $g(\X)$, so 
    $$
        H(P) = \langle 2q/3, q, q+1\rangle.
    $$
    It is easy to see that the  semigroup $H(P)$ is telescopic (for the definition see \cite[Definition 6.1]{KP}) and hence symmetric from \cite[Lemma 6.5]{KP}. Thus
    $$
        W \sim (2g-2)P,
    $$
    where $W$ is a canonical divisor. By Lemma \ref{lemma:divisor_A} this means
    $$
        (2g-2)P \sim A + \left( \frac{q}{3}-1\right)D.
    $$
    Since $P$ is $\Fqq$-rational we have $D = (q+1)P_0 \sim (q+1)P$, so the above implies
    $$
        \left(\frac{q}{3}-1\right)P = \left(2g-2 - \left(\frac{q}{3}-1\right)\cdot(q+1)\right)P \sim A.
    $$
    From Lemma \ref{lemma:divisor_A} we know that $A$ is the only effective divisor in its divisor class, so we conclude that
    $$
        A = \left( \frac{q}{3}-1\right) P. 
    $$
\end{proof}

\section{The cubic and the defining equation}

The goal of this section is to prove Theorem \ref{thm:main}. We will do so by showing that the function field $F$ of $\X$, as defined in the previous section, is isomorphic to the function field of the known curve with the third largest genus. This is sufficient since any function field can be realized as the function field of a non-singular projective curve (see, e.g., \cite[Appendix B.9]{stichtenothAlgebraicFunctionFields2009}). We keep the assumptions and notation from the previous sections.

\subsection{The shape of the cubic and the Hermitian surface}

From now on, we denote by $P$ the place of $F$ that corresponds to the special $\Fqq$-rational point from Corollary \ref{cor:special_place}. We consider the Frobenius embedding $\pi:\X \to \mP^3$ defined by $(1:u:v:w)$ where $u,v,w \in \mathcal{L}((q+1)P)$ with
$$
    (u)_\infty = \frac{2q}{3}P, \quad (v)_\infty= qP, \quad \text{and} \quad (w)_\infty = (q+1)P.
$$
With notation as in Section \ref{sec:preliminaries} this means that we are taking $P_0 = P$, $D=(q+1)P$ and then $\pi$ is the morphism associated with $\mathcal{D} = |D|$. Note however, that our choice of basis for $\mathcal{L}(D)$ means that we can no longer assume that the Hermitian surface containing $\pi(\X)$ is in its canonical form. The main goal of this section is to understand better both the shape of the cubic relation given by Proposition \ref{prop:contained_in_cubic} and the shape of the Hermitian surface containing $\pi(\X)$. First a result regarding the cubic: 

\begin{lemma}
    We can choose $u, v$ and $w$ such that the cubic relation is 
    $$
        v^2 = au^3 + (bu + c)w + du^2 + eu + f, 
    $$
    for some $a,b,c,d,e,f \in \Fqq$ and $a\neq 0$.
\end{lemma}

\begin{proof}
Since $H(P)=\langle 2q/3,q,q+1 \rangle$, the eight smallest elements of this semigroup are given by $0,2q/3,q,q+1,4q/3,5q/3,5q/3+1$ and $2q$. Of these eight elements, only $2q$ can be written in two ways in terms of the three generators of $H(P)$, namely $2q=3\cdot 2q/3$ and $2q=2\cdot q$. This shows that there exists a cubic relation between $u,v$ and $w$ of the form
%
    $$
        v^2 + a_1u^3 +a_2uw+a_3uv+a_4u^2+a_5w+a_6v+a_7u + a_8 = 0,
    $$
    with $a_i \in \Fqq$ and $a_1 \neq 0$. Completing the square in $v$ now gives the shape we claimed (with $v$ replaced by $v+a_3/2 \cdot u + a_6/2$).
\end{proof}

In order to study the image of $\pi = (1:u:v:w)$ at $P$ it is useful to change coordinates. From now on, let $t := v/w$, $\xi := u/w$ and $\zeta := 1/w$, so that
$$
    v_P(\zeta)=q+1, \quad v_P(\xi) = \frac{q}{3}+1 \quad \text{and} \quad v_P(t)=1.
$$
Note in particular that $t$ is a local parameter at $P$ and $\pi = (\zeta: \xi:t:1)$ maps $P$ to $[0:0:0:1]$. We obtain the following result about the Hermitian surface containing $\pi(\X)$:

\begin{lemma}
The functions $\zeta, \xi$ and $t$ satisfy
    $$
        (\zeta^q \ \xi^q \ t^q \ 1) \ H 
        \begin{pmatrix}
            \zeta \\
            \xi \\
            t \\
            1
        \end{pmatrix} = 0,
    $$
    where
    $$
        H := 
        \begin{pmatrix}
            \rho & \mu^q & \nu^q & \eta^q \\
            \mu & \alpha & \kappa^q & 0 \\
            \nu & \kappa & \beta & 0 \\
            \eta & 0 & 0 & 0
        \end{pmatrix}.
    $$
Here $\alpha, \beta, \rho \in \Fq$ and $\eta,\kappa,\mu,\nu \in \Fqq$. Moreover, $\eta\beta \neq 0$. Equivalently, the Hermitian surface containing $\pi(\X)$ is given by the equation
    $$
        (X_0^q \ X_1^q \ X_2^q \ X_3^q) \ H 
        \begin{pmatrix}
            X_0 \\
            X_1 \\
            X_2 \\
            X_3
        \end{pmatrix} = 0.
    $$
\end{lemma}

\begin{proof}
    We know that $\pi(\X)$ is contained in a non-degenerate Hermitian surface, say given by the equation $G_H=0$ where 
    $$
        G_H:= (X_0^q \ X_1^q \ X_2^q \ X_3^q) \ H 
        \begin{pmatrix}
            X_0 \\
            X_1 \\
            X_2 \\
            X_3
        \end{pmatrix} \in \Fqq[X_0,X_1,X_2,X_3],
    $$   
    and
    $$
        H := 
        \begin{pmatrix}
            \rho & \mu^q & \nu^q & \eta^q \\
            \mu & \alpha & \kappa^q & \lambda_1^q \\
            \nu & \kappa & \beta & \lambda_2^q \\
            \eta & \lambda_1 & \lambda_2 & \lambda_3
        \end{pmatrix}        
    $$
    with $\alpha, \beta, \rho,\lambda_3\in \Fq$ and $\eta,\kappa,\mu,\nu,\lambda_1,\lambda_2 \in \Fqq$. The claim is that $\lambda_i=0$ for $i=1,2,3$ and $\eta\beta \neq 0$.

    Note that $\pi(P) =[0:0:0:1]$. The osculating hyperplane of $\pi(\X)$ at $P$ is given by 
    $$
        X_0 = 0,
    $$
    since $v_P(a_0\zeta+a_1\xi + a_2t + a_3)$ is largest at $a_1=a_2=a_3=0$. On the other hand, the tangent hyperplane of the non-degenerate Hermitian variety at $P$ is given by
    $$
        0 = \sum_{i=0}^3 \left.\frac{\partial G_H}{\partial X_i}\right\vert_{\pi(P)} X_i= \eta X_0 + \lambda_1X_1 + \lambda_2X_2 + \lambda_3X_3 
    $$
    The osculating hyperplane coincides with this tangent hyperplane (see the paragraph after Theorem 3.6 in \cite{Korchmáros_Torres_2001}), so we get $\lambda_i = 0$ for $i=1,2,3$ and $\eta \neq 0$.

    To show $\beta \neq 0$, we consider the terms in $G_H(\zeta, \xi, t,1)$ with the smallest possible valuation at $P$. There must be at least two such terms to get the required cancellation. Since $\lambda_i =0$ there is no term with valuation less than $q+1$, and $\eta \neq 0$ implies that there is at least one term with this valuation, namely $\eta \cdot \zeta$. The only other term that can have valuation $q+1$ is $\beta \cdot t^{q+1}$, so we conclude also $\beta \neq 0$ as wished.
\end{proof}

The Hermitian relation can be written as 
$$
    \eta \zeta + \beta t^{q+1} +\kappa t^q\xi + \nu t^q \zeta + \kappa^q t \xi^q + \alpha \xi^{q+1}+ \mu \xi^q \zeta + \eta^q \zeta^q + \nu^q  t\zeta^q + \mu^q \xi \zeta^q + \rho \zeta^{q+1} = 0.
$$
In particular, we have
\begin{equation}\label{eq:zeta_local_expansion}
    -\eta \zeta = \beta t^{q+1} + \kappa t^q\xi + \mathcal{O}_P\left(t^{2q+1}\right), 
\end{equation}
where the notation $x = y + \mathcal{O}_P(z)$ means $v_P(x-y) \geq v_P(z)$. After dividing with $w^3$ the cubic relation in terms of $\zeta$, $\xi$ and $t$ reads
\begin{equation}\label{eq:cubic_zeta}
    t^2\zeta = a \xi^3 + b\xi \zeta + c \zeta^2+d\xi^2\zeta + e \xi \zeta^2 + f \zeta^3.
\end{equation}

\begin{lemma}
    We have $b=0$, $\kappa = 0$ and $\alpha \neq 0$.
\end{lemma}

\begin{proof}
    By Equation \eqref{eq:zeta_local_expansion} the local expansion at $P$ in the local parameter $t$ of the left hand side of Equation \eqref{eq:cubic_zeta} has no term of order $q+\frac{q}{3}+2$. The local expansion of $a\xi^3$ only contains terms of order divisible by $3$, in particular it has no term of order $q+\frac{q}{3} + 2$ either. If $b\neq 0$ then $v_P(b\xi\zeta) = q + \frac{q}{3} + 2$, while all other terms on the right hand side in Equation \eqref{eq:cubic_zeta} have valuation strictly larger than $q+\frac{q}{3}+2$. We conclude that $b=0$.

    Next we claim that the local expansion of the right-hand side of Equation \eqref{eq:cubic_zeta} has no term of order $q + \frac{q}{3} + 3$. Since by Equation \eqref{eq:zeta_local_expansion}, the left-hand side of Equation \eqref{eq:cubic_zeta} contains a term of order $q + \frac{q}{3} + 3$ if and only if $\kappa \neq 0$, this claim would imply $\kappa=0$.
    
    If the local expansion of the right hand side of Equation \eqref{eq:cubic_zeta} has a term of order $q + \frac{q}{3} + 3$, it have to come from $a\xi^3$, so we proceed by studying the local expansion
    $$
        \xi = \sum_{n\geq \frac{q}{3} + 1} a_nt^n, \quad \text{with} \quad a_n \in \Fqq, a_{\frac{q}{3}+1} \neq 0.  
    $$
    From the cubic equation and $b=0$ we get
    $$
        t^2 \zeta = a \xi^3 + \mathcal{O}_P\left(t^{q+2\frac{q}{3} + 3}\right),
    $$
    while Equation \eqref{eq:zeta_local_expansion} implies that
    $$
        t^2 \zeta = \frac{-\beta}{\eta}t^{q+3} + \frac{-\kappa}{\eta}t^{q+2}\xi + \mathcal{O}_P(t^{2q+3}).
    $$
    Comparing the two yields
    \begin{equation}\label{eq:xi_cubed_expansion}
         \xi^3 = \frac{-\beta}{a\eta}t^{q+3} + \frac{-\kappa}{a\eta}t^{q+2}\xi + \mathcal{O}_P\left(t^{q+2\frac{q}{3} + 3}\right).
    \end{equation}
    Note that we recover our previous observation: it holds that $\kappa = 0$ if and only if the local expansion of $\xi^3$ at $P$ has no non-zero term of order $q + \frac{q}{3} + 3$, i.e., if and only if $a_{\frac{q}{3} + \frac{q}{9} + 1} =0$.

    The key to showing that $a_{\frac{q}{3} + \frac{q}{9} + 1} =0$ turns out to be the results we have on the ramification divisor $R$. With notation as in the previous sections, recall that
    $$
        v_P(R) = v_P(E+T) = v_P(3A + T) = q-2.
    $$
    By the definition of the ramification divisor, \cite[Proposition 1.4]{stohrWeierstrassPointsCurves1986} and the fact that $t$ is a local parameter of $t$, and hence also a separating element of $F$, we have
    \begin{align*}
        v_P(R) &= v_P\left( \det 
        \begin{pmatrix}
            1 & u & v & w \\
            D_t^{(1)} 1 & D_t^{(1)}u & D_t^{(1)}v & D_t^{(1)}w \\
            D_t^{(2)}1 & D_t^{(2)}u & D_t^{(2)}v & D_t^{(2)}w \\
            D_t^{(q)}1 & D_t^{(q)}u & D_t^{(q)}v & D_t^{(q)}w
        \end{pmatrix}
        \right)+ 4(q+1) \\
        &= v_P\left( w^4 \det 
        \begin{pmatrix}
            \zeta & \xi & t & 1 \\
            D_t^{(1)}\zeta & D_t^{(1)}\xi & D_t^{(1)}t & D_t^{(1)}1 \\
            D_t^{(2)}\zeta & D_t^{(2)}\xi & D_t^{(2)}t & D_t^{(2)}1 \\
            D_t^{(q)}\zeta & D_t^{(q)}\xi & D_t^{(q)}t & D_t^{(q)}1
        \end{pmatrix}
        \right)+ 4(q+1) \\
        &= v_P \left(\det 
        \begin{pmatrix}
            \zeta & \xi & t & 1 \\
            D_t^{(1)}\zeta & D_t^{(1)}\xi & 1 & 0 \\
            D_t^{(2)}\zeta & D_t^{(2)}\xi & 0 & 0 \\
            D_t^{(q)}\zeta & D_t^{(q)}\xi & 0 & 0
        \end{pmatrix}
        \right) \\
        &= v_P \left(  D_t^{(2)}\zeta \cdot  D_t^{(q)}\xi - D_t^{(2)}\xi \cdot D_t^{(q)}\zeta
        \right).
    \end{align*}

    Since $v_P(D_t^{(2)} \zeta) \geq q-1$ and $v_P(R)=q-2$, we see that
    $$
        v_P(D_t^{(2)}\xi \cdot D_t^{(q)}\zeta) = q-2.
    $$
    By Equation \eqref{eq:zeta_local_expansion} we have $v_P(D_t^{(q)} \zeta) = 1$ so 
    $$
        v_P(D_t^{(2)}\xi) = q-3.
    $$
    On the other hand, we have 
    $$
        D_t^{(2)}\xi = D_t^{(2)}\left( \sum_{n\geq \frac{q}{3}+1} a_nt^n\right) = a_{\frac{q}{3}+1} \cdot \binom{\frac{q}{3}+1}{2} \cdot  t^{\frac{q}{3}-1}+ a_{\frac{q}{3}+2}\cdot  \binom{\frac{q}{3}+2}{2}\cdot t^{\frac{q}{3}} + \cdots.
    $$
    It follows that 
    \begin{equation}\label{eq:a_n_equal_zero}
        a_{n} = 0 \quad \text{for all} \quad n \equiv 2\pmod 3 \quad \text{with} \quad n \leq q-4. 
    \end{equation}
    We will use this, and the local information from the cubic and Hermitian relations, to show that $\kappa = 0$. In fact, suppose by contradiction that $\kappa \neq 0$. Then $\xi^3$ has a term of order $q + \frac{q}{3} + 3$ by Equation \ref{eq:xi_cubed_expansion}). Recalling the characteristic of the base field is $3$, this implies that $\xi$ has a term of order $\frac{q}{3} + \frac{q}{9} + 1$, i.e., 
    $$
        a_{\frac{q}{3} + \frac{q}{9} + 1} \neq 0.
    $$
      By Equation (\ref{eq:xi_cubed_expansion}) the above implies that $\xi^3$ has a term of order $\frac{q}{3} + \frac{q}{9} + 1 + (q+2) = q + \frac{q}{3} + \frac{q}{9} + 3$, which in turn means that $\xi$ has a term of order $\frac{q}{3} + \frac{q}{9} + \frac{q}{27} + 1$. After a finite number of repetitions we get that 
    $$
        a_{\frac{q}{3} + \frac{q}{9} + \cdots + \frac{q}{q} + 1} \neq 0,
    $$
    which contradicts Equation \eqref{eq:a_n_equal_zero} since $\frac{q}{3} + \frac{q}{9} + \cdots + \frac{q}{q} + 1 \equiv 2 \pmod 3$ and 
    $$
        \frac{q}{3} + \frac{q}{9} + \cdots + \frac{q}{q} +  1 = \frac{q+1}{2}  \leq q-4,
    $$
    for $q\geq 9$. We conclude that $\kappa = 0$.
    
    Finally, since the Hermitian variety is non-degenerate we get that $0 \neq \det H = -\eta^{q+1}\alpha \beta$, so we also have $\alpha \neq 0$ as claimed.
\end{proof}
At this point the obtain Hermitian relation between the functions $t$, $\zeta$ and $\xi$ is
\begin{equation} \label{eq:Herm_k=0}
      \eta \zeta + \beta t^{q+1} + \nu t^q \zeta + \alpha \xi^{q+1}+ \mu \xi^q \zeta + \eta^q \zeta^q + \nu^q  t\zeta^q + \mu^q \xi \zeta^q + \rho \zeta^{q+1} = 0.  
\end{equation}

Looking instead at the shape of the cubic relation, our last and more difficult step in this section is showing that  $c,d \neq 0$. Once we have $c\neq 0$ we can express $w$ as a polynomial in $u$ and $v$, which implies in particular that $F = \Fqq(u,v)$.

\begin{proposition}
    We have $c \neq 0$ and $d \neq 0$.
\end{proposition}


\begin{proof}
    Suppose by contradiction that $c = 0$, so that the cubic relation reads
    $$
        v^2 = au^3 + du^2 + eu+f.
    $$
    Consider the function field $E_0 := \Fqq(u,v) \subseteq F$. We claim that $E_0$ is an elliptic function field. If not, the polynomial $aT^3 + dT^2 + eT+f \in \Fqq[T]$ has a repeated root, i.e,
    $$
        au^3 + du^2 + eu + f = a(u-r_1)^2(u-r_2),
    $$
    for some not necessarily distinct $r_1,r_2 \in \Fqq$ (see \cite[Proposition 6.1.3]{stichtenothAlgebraicFunctionFields2009}). This would mean that 
    $$
        h := \frac{v}{u-r_1} \in \mathcal{L} \left(\frac{q}{3}P\right) \subseteq F,
    $$
    since 
    $$
        h^2 = a(u-r_2)\in \mathcal{L}\left(\frac{2q}{3}P\right),
    $$
    in contradiction with $H(P) = \langle 2q/3, q, q+1\rangle$. We conclude that $E_0$ is elliptic. Moreover, considering the ramification of the pole of $u$ in the extension $E_0/\Fqq(u)$, one can conclude that $[F:E_0] = q/3$. Further, maximality of $F$ implies that $E_0$ is maximal, so $E_0$ is supersingular and it follows from \cite[V Theorem 4.1]{silvermanArithmeticEllipticCurves2009} that $d=0$.

    After multiplying Equation \eqref{eq:Herm_k=0} with $w^{q+1}$ we get
    $$
        \eta w^q + \eta^q w = L(u,v) \in \Fqq[u,v],
    $$
    where
    $$
        L(u,v) = - \left(\beta v^{q+1} + \alpha u^{q+1} + \nu v^q + \mu u^q  + \nu^q v + \mu^q u + \rho\right).
    $$
    Note that $\eta T^q + \eta^qT - L(u,v) \in E_0[T]$ is reducible since $[F:E_0] = q/3$, so the minimal polynomial of $w$ over $E_0$ must be  a proper factor of it. The additive separable polynomial $\eta T^q + \eta^q T \in \Fqq[T]$ has all its roots in $\Fqq$. In fact, if $\theta \neq 0$ is a root then
    $$
        \theta^{q^2-1} = \left(\theta^{q-1}\right)^{q+1} = \left(-\eta^{q-1}\right)^{q+1} = \eta^{q^2-1} = 1. 
    $$
    It follows that $F$ is the splitting field over $E_0$ of the separable polynomial $\eta T^q + \eta^q T - L(u,v)$, in particular $F/E_0$ is a Galois extension. The Galois group is an elementary abelian group which can be described as
    $$
        \mathrm{Gal}(F/E_0) = \{ \sigma_\theta: w \mapsto w+\theta \mid \theta \in V\},
    $$
    where $V$ is some additive subgroup of order $q/3$ of the roots of $\eta T^q + \eta^q T$.

    Let $V_1$ be a fixed subgroup of $V$ with $|V/V_1| = 3$ and let $G_1$ be the corresponding subgroup of $\mathrm{Gal}(F/E_0)$, i.e.,
    $$
        G_1 := \{\sigma_\theta \mid \theta \in V_1\} \subseteq \mathrm{Gal}(F/E_0).
    $$
    The polynomial 
    $$
        L_1 := \prod_{\theta \in V_1} (T-\theta) \in \Fqq[T]
    $$
    is additive (see \cite[Theorem 3.52]{lidlFiniteFields2008}) and $z := L_1(w)$ is fixed by $G_1$. Our goal is to show that the function field 
    $$
        E_1 := E_0(z)\subseteq F
    $$
    cannot be maximal, which is a contradiction since it is a subfield of $F$. 
    
    First, we note that $E_1$ is the whole fixed field $F^{G_1}$ of $G_1$. Clearly $E_1$ is contained in $F^{G_1}$, whence $[F:E_1] \ge |V_1|=q/9$. On the other hand, $F=E_1(w)$ and $L_1(w) \in E_1$, whence $[F:E_1] \le \deg L_1=|V_1|$.

    Now pick an element $\theta_0 \in V\setminus V_1$. The Galois group of $E_1/E_0$ is generated by the image of $\sigma_{\theta_0}$ in $\mathrm{Gal}(F/E_0)/G_1$. Moreover, 
    $$
        \sigma_{\theta_0}(z) = \prod_{\theta \in V_1} (w+\theta_0 + \theta ) = L_1(w+\theta_0) = z + L_1(\theta_0),
    $$
    by additivity of $L_1$. After replacing $z$ by $z/L_1(\theta_0)$ we may therefore assume that $\mathrm{Gal}(E_1/E_0)$ maps $z$ to $z+1$ and $z+2$, which means that a defining equation for $E_1$ over $E_0$ is
    $$
        z^3 - z = h
    $$
    for some $h \in E_0$.

    Let $Q_0$ be the unique place lying below $P$ in $E_0$, i.e., the common pole of $u$ and $v$ in $E_0$, and let $Q_1$ be the place lying above $Q_0$ (and below $P$) in $E_1$. Then, since $v_P(w)=-(q+1)$ and $z=L_1(w)$, we have
    $$
        v_{P}(z) = -|V_1|\cdot (q+1).
    $$
    As $z$ has no other poles in $F$, since $w$ has no other poles in $F$, it follows that $Q_1$ is the only pole of $z$ in $E_1$ and that
    $$
        v_{Q_1}(z) = -(q+1).
    $$
    This in turn implies that
    \begin{equation}\label{eq:valuation_h}
        v_{Q_0}(h) = -(q+1)        
    \end{equation}
    and that $Q_0$ is the only pole of $h$. We conclude that
    $$
        \mathrm{Diff}(E_1/E_0) = 2(q+2)Q_1
    $$
    (see \cite[Proposition 3.7.8]{stichtenothAlgebraicFunctionFields2009}).

    Recall that $q = 3^n \geq 9$, and that the $n$'th power of the Cartier operator, $\C$, maps all regular differentials of an $\mathbb{F}_{p^{2n}}$-maximal function field to zero (see \cite[Theorem 3.3]{garciaCertainMaximalCurves2008}). We would like to show that the differential $z \omega$, with $\omega := du/v$, is regular and that $\C^n(z\omega) \neq 0$. \newline

    \emph{Claim 1:} The differential $z\omega$ is regular. \newline

    It is well known that $\omega = du/v$ is a regular differential on $E_0$; the divisor of $\omega$ in $E_0$ is $0$ (see the comment after the proof of Proposition 6.1.3 in \cite{stichtenothAlgebraicFunctionFields2009}). Moreover, since the space of regular differentials on an elliptic function field has dimension $1$ we get from maximality of $E_0$ and \cite[Theorem 3.3]{garciaCertainMaximalCurves2008} that $\C(\omega) = 0$. 

    It follows from \cite[Theorem 3.4.6]{stichtenothAlgebraicFunctionFields2009} that the divisor of $\omega$ when viewed as a differential on $E_1$ is 
    $$
        (\omega)_{E_1} = \mathrm{Diff}(E_1/E_0) = 2(q+2)Q_1.
    $$
    It follows that
    $$
        v_{Q_1}(z\omega) =  -(q+1) + 2(q+2) = q + 3,
    $$
    and since $z$ has no other poles we get that $z\omega$ is a regular differential. \newline

    \emph{Claim 2:} We have $\C^n(z\omega) \neq 0$. \newline

    From $z = z^3-h$ we get 
    $$
        \C(z\omega) = \C(z^3\omega - h\omega) = z\C(\omega) - \C(h\omega) = -\C(h\omega),
    $$
    by properties of the Cartier operator (see, e.g., \cite[Exercise 4.12]{stichtenothAlgebraicFunctionFields2009}), and it follows that
    $$
        \C^n(z\omega) = - \C^n(h\omega).
    $$
    The restriction of $\C$ on $E_1$ to $E_0$ is exactly the Cartier operator on $E_0$, so the above means that we just have to show $\C^n(h\omega) \neq 0$, where $h\omega$ is seen as a differential on $E_0$. 

    Let $t_0$ be a local parameter at $Q_0$. We consider the expansion
    $$
        h \omega = \sum_{-(q+1)\leq i} a_i t_0^i dt_0,
    $$
    where $a_i \in \Fqq$ and $a_{-(q+1)}\neq 0$ using Equation \eqref{eq:valuation_h}. For integers $m$ and $k$ with $0\leq k \leq q-1$ it follows from properties of the Cartier operator (see \cite[Remark 3.1]{garciaCertainMaximalCurves2008}) that 
    $$
        \C^n(t_0^{mq+k}dt_0) = 
        \begin{cases}
            0 & \text{ if } k < q-1 \\
            t_0^m dt_0 & \text{ if } k = q-1.
        \end{cases}
    $$
    Hence,
    $$
        \C^n(h\omega) = \left( a_{-(q+1)}^{1/q} t_0^{-2} + a_{-1}^{1/q}t_0^{-1} + a_{q-1}^{1/q} + \text{ higher order terms} \right) dt_0,
    $$
    and since $a_{-(q+1)} \neq 0$ we obtain that $C^n(h\omega) \neq 0$ as wished. The claim follows and we obtain a contradiction, so we conclude that $c \neq 0$. \newline

    The only thing left to show is that $d \neq 0$. To see this, we compare once more the Hermitian and cubic relation. From $c\neq0$ we get that $v_P(c\zeta^2)=2q+2$, and the only other terms in the cubic relation \eqref{eq:cubic_zeta} (with $b=0$) that might have a local power series expansion at $P$ with a non-zero term of order $2q+2$ are $d\xi^2\zeta$ and $t^2\zeta$. However, one can see from the Hermitian relation \eqref{eq:Herm_k=0} that the local expansion of $\zeta$ has no term of order $2q$, so we must have $d\neq0$. This finishes the proof. 
\end{proof}

The results above can be restated as follows:

\begin{corollary}\label{cor:cubic_hermitian_shape}
    We can pick $u,v,w \in \mathcal{L}(D)$ with
    $$
        v_P(u) = -\frac{2q}{3}, \quad v_P(v) = -q, \quad \text{and} \quad v_P(w) = -(q+1)
    $$
    such that
    $$
        w = u^3 + \delta u^2 -v^2,
    $$
    for some $\delta \in \Fqq^\times$. Moreover, the image of $\X$ under the map $\pi := (1:u:v:w)$ is contained in the Hermitian surface defined by
    $$
        H := 
        \begin{pmatrix}
            \rho & \mu^q & \nu^q & \eta^q \\
            \mu & \alpha & 0 & 0 \\
            \nu & 0 & \beta & 0 \\
            \eta & 0 & 0 & 0
        \end{pmatrix},
    $$
    where $\alpha, \beta, \rho \in \Fq$, $\eta,\mu,\nu \in \Fqq$ and $\alpha\beta\eta \neq 0$.
\end{corollary}
    
\begin{proof}
    This is just changing $u$ by a scalar $\lambda$ with $\lambda^3=a$ and replacing $w$ with $-(cw+eu+f)$. The Hermitian matrix changes accordingly but remains of the same sparse shape since the only terms involving $w$ were $\eta w^q$ and $\eta^qw$ for which the change of coordinates only results in terms with exponents that were already there.
\end{proof}

\subsection{The defining equation}

Since $\pi = (1:u:v:w)$ is an embedding we have $F = \Fqq(u,v,w)$, which we will use as our starting point. It follows immediately from Corollary \ref{cor:cubic_hermitian_shape} that $w\in\Fqq[u,v]$, so in fact 
$$
    F = \Fqq(u,v,w) = \Fqq(u,v).
$$
We will proceed by studying the extension $F/\Fqq(v)$.

\begin{proposition}\label{prop:different_F_over_K_v}
    We have 
    $$
        \div (dv) = (2g-2)P,
    $$
    and the extension $F/\Fqq(v)$ is separable of degree $q$, with
    $$
        \operatorname{Diff}(F/\Fqq(v)) =  \left(\frac{q^2+5q-6}{3}\right) P.
    $$
    In particular, the pole of $v$ is totally ramified in $F/\Fqq(v)$ and it is the only ramified place.
\end{proposition}

\begin{proof}
    The extension $F/\Fqq(v)$ has degree $q$ since $(v)_\infty = qP$ (see \cite[Theorem 1.4.11]{stichtenothAlgebraicFunctionFields2009}), so the pole of $v$ is totally ramified. Moreover, $dv \neq 0$ since $v$ is not a cube (otherwise $q/3$ would be a non-gap at $P$, in contradiction with Corollary \ref{cor:special_place}).

    The cubic and the Hermitian relation from Corollary \ref{cor:cubic_hermitian_shape} imply that
    $$
        dw = -\delta u du + vdv,
    $$
    and 
    $$
        (\alpha u^q + \mu^q)du + (\beta v^q+\nu^q)dv + \eta^q dw = 0.
    $$
    Eliminating $du$, we find that
    $$
        (\delta\beta u v^q + (\alpha u^q + \mu^q)v + \delta\nu^qu) dv = (\alpha u^q - \eta^q \delta u + \mu^q)dw.
    $$
    Since $v_P(w) = -(q+1)$ we get immediately 
    that $v_P(dw) = -(q+2)$, so the above implies
    \begin{align*}
        v_P(dv) &= -v_P(\delta\beta u v^q + (\alpha u^q + \mu^q)v + \delta\nu^qu) + v_P(\alpha u^q - \eta^q \delta u + \mu^q) - (q+2) \\
        &= \frac{2q}{3} + q^2 - \frac{2q^2}{3} - (q+2) \\
        &= \frac{q^2-q-6}{3} \\
        &= 2g-2.
    \end{align*}
    Since the degree of the canonical divisor $(dv)$ is exactly $2g-2$, and since $v$ has no poles away from $P$, we conclude that $(dv) = (2g-2)P$ as wished. 

    From \cite[Remark 4.3.7]{stichtenothAlgebraicFunctionFields2009} we get that
    $$
        \mathrm{Diff}(F/\Fqq(v)) = (dv) + 2(v)_\infty = \left(\frac{q^2+5q-6}{3}\right) P,
    $$
    and it follows that the pole of $v$ is the only ramified place in $F/\Fqq(v)$.

\end{proof}

The first step towards describing the defining equation of $F$ is the following lemma:

\begin{lemma}\label{lem:integralbasis}
    The minimal polynomial of $u$ over $\Fqq(v)$ is a polynomial over $\Fqq[v]$ of degree $q$. Moreover,
    $$
        1, u, u^2, \dots, u^{q-1}
    $$
    is an integral basis of $F/\Fqq(v)$ at any place $Q$ of $\Fqq(v)$ different from the pole of $v$.
\end{lemma}

\begin{proof}
    Since $\Fqq[v]=\bigcap_{Q \ne Q_\infty} \mathcal{O}_Q$, where $\mathcal{O}_Q$ denotes the local ring at $Q$ and $Q_\infty$ denotes the pole of $v$, $\Fqq[v]$ is a holomorphy ring of $\Fqq(v)$ (in the sense of \cite[Definition 3.2.2]{stichtenothAlgebraicFunctionFields2009}). Furthermore 
    $$
        H(P) = \langle -v_P(u),-v_P(v), -v_P(w)\rangle,
    $$ 
    which together with \cite[Theorem 3.2.6]{stichtenothAlgebraicFunctionFields2009} implies that the integral closure of $\Fqq[v]$ in $F$ is $\Fqq[u,v,w] = \Fqq[u,v]$. Now \cite[Proposition 3.3.1]{stichtenothAlgebraicFunctionFields2009} implies that the minimal polynomial $\varphi(T) \in \left(\Fqq(v)\right)[T]$ of $u$ is a polynomial over $\Fqq(v)$, i.e.,
    $$
        \varphi(T) \in \left(\Fqq[v]\right)[T],
    $$
    and $\deg \varphi = [F:\Fqq(v)] = q$. It follows that $1, u, u^2, \dots, u^{q-1}$ is an integral basis of $F/\Fqq(v)$ at any place different from the pole of $v$ (see \cite[Exercise 3.4]{stichtenothAlgebraicFunctionFields2009}).
\end{proof}

We finish the proof of the main theorem in three steps now. In the first step we determine the non-constant terms of the minimal polynomial of $u$:

\begin{proposition}\label{prop:shape_of_phi}
    The function field $F = \Fqq(u,v)$ can be defined by an equation of the form
    $$
        u^q + \gamma u + G(v) = 0,
    $$
    where $\gamma :=-\frac{\delta \eta^q}{\alpha} \in \Fqq^{\times}$ and $G(v)\in \Fqq[v]$.
\end{proposition}
\begin{proof}
    Let $\varphi(T) \in \left(\Fqq[v]\right)[T]$ be the minimal polynomial of $u$ over $\Fqq(v)$. For any place $Q_0$ of $\Fqq(v)$ different from the pole of $v$, and any $Q$ lying above $Q_0$ in $F$ we have 
    $$
        v_{Q}(\varphi'(u)) = d(Q \mid Q_0) = 0
    $$
    by Proposition \ref{prop:different_F_over_K_v}, Lemma \ref{lem:integralbasis} and \cite[Theorem 3.5.10 (b)]{stichtenothAlgebraicFunctionFields2009}. It follows that $\varphi'(u)$ is a constant, i.e., 
    $$
        \varphi'(u) = \gamma \in \Fqq^{\times}.
    $$
    Since $\deg(\varphi'(T)-\gamma) <q$ this implies that $\varphi'(T) = \gamma$, and $\varphi$ is monic, so we must have 
    $$
        \varphi(T) = T^q + \gamma \cdot T + \sum_{i=0}^{q/3-1} A_{i} \cdot  T^{3i},
    $$
    where each $A_i$ is a polynomial in $v$, i.e., $A_i \in \Fqq[v]$. For convenience, we define the polynomial
    $$
        S(T) := \sum_{i=0}^{q/3-1} A_i' \cdot T^{3i} \in \left(\Fqq[v]\right)[T],
    $$
    where $A_i'$ denotes the derivative of $A_i$ with respect to $v$.
    
    From $d\varphi(u) = 0$ we then get
    $$
        \gamma du +S(u) dv = 0.
    $$
    On the other hand, after differentiating the cubic and the Hermitian equation, as in the proof of Proposition \ref{prop:different_F_over_K_v}, and eliminating $dw$ we get
    $$
        (\alpha u^q - \eta^q \delta u + \mu^q) du + (\beta v^q + \eta^q v +\nu^q)dv = 0.
    $$
    Combining the two equations yields
    $$
        (\alpha u^q - \eta^q \delta u + \mu^q) S(u) dv = \gamma (\beta v^q + \eta^q v +\nu^q)dv
    $$

    Since $dv \neq 0$ this means that
    \begin{equation}\label{eq:differentiated_hermitian}
        (\alpha u^q - \eta^q \delta u + \mu^q) S(u) - \gamma (\beta v^q + \eta^q v +\nu^q) = 0.
    \end{equation}
    When viewed as a polynomial in $u$ the left hand side must be divisible by the minimal polynomial of $u$, i.e. there exists some polynomial $\psi(T) \in \left(\Fqq[v]\right)[T]$ such that
    \begin{equation}\label{eq:diff_Su} 
        (\alpha T^q - \eta^q \delta T + \mu^q) S(T) - \gamma (\beta v^q + \eta^q v +\nu^q)  = \varphi(T)\cdot\psi(T).
    \end{equation}
    Note that the degree of $\psi(T)$ is at most $q-3$. Differentiating with respect to $T$, using $\varphi'(T) = \gamma$ and rearranging yields
    $$
        -\eta^q \delta \cdot S(T) -\gamma\cdot\psi(T) =  \varphi(T)\cdot \psi'(T).
    $$
    Since the right hand side is divisible by $\varphi(T)$ and has degree strictly less than $q$ it must be zero, i.e.,
    $$
        \psi(T) = - \frac{\eta^q\delta}{\gamma} S(T).
    $$
    By inserting this in Equation \eqref{eq:diff_Su} and rearranging we get
    $$
        S(T) \cdot \left(\alpha T^q - \eta^q \delta T + \mu^q + \frac{\eta^q\delta}{\gamma}\varphi(T) \right) = \gamma (\beta v^q + \eta^q v +\nu^q) \in \Fqq[v],
    $$
    which means that both factors on the left must have $T$-degree zero. Since $\varphi(T) = T^q + \gamma \cdot T + \sum_{i=0}^{q/3-1} A_{i} \cdot  T^{3i}$ we conclude that
    $$
        \alpha + \frac{\delta \eta^q}{\gamma} = 0 \quad \text{and} \quad \sum_{i=0}^{q/3-1}A_iT^{3i} = A_0.   
    $$
    In particular, with $G(v) := A_0 \in \Fqq[v]$ we get 
    $$
        u^q +\gamma u + G(v) = 0
    $$
    and $\gamma = -\frac{\delta \eta^q}{\alpha}$ as claimed.
\end{proof}


The second step is to study $G(v)$ in more detail.

\begin{lemma}\label{lemma:G_v}
    The polynomial $G(v)$ has degree $2q/3$, the leading coefficient is $-1$, and it holds that 
    $$
        G'(v)\left(\mu^q - \alpha G(v)\right) = \gamma(\beta v^q + \eta^q v + \nu^q).
    $$
    Moreover, the polynomial $G'(v)$ has degree $q/3$, the leading coefficient is $\ell := \gamma\beta/\alpha$ and 
    $$
        G(v) = -\left(\frac{G'(v)}{\ell}\right)^2 + \theta,
    $$
    where $\theta := \frac{\mu^q}{\alpha} - \frac{\gamma \eta^q}{\alpha \ell^2} = \frac{\mu^q}{\alpha} + \frac{\alpha^2}{\delta\beta^2} \in \Fqq$.
\end{lemma}

\begin{proof}
    To see that $\deg G = 2q/3$ we consider the valuation at $P$ of $u^q + \gamma u  = - G(v)$: the left hand side has valuation $-2q^2/3$ and the right hand side has valuation $-q\cdot \deg(G(v))$. The leading coefficient of $G(v)$ can be determined in a similarly simple manner: From the cubic equation we get
    $$
        \frac{u^3}{v^2} = \frac{w-\delta u^2}{v^2} + 1,
    $$
    and
    $$
        v_P\left(\frac{w-\delta u^2}{v^2}\right) = 2q - \frac{4q}{3} = \frac{2q}{3}>0, 
    $$
    so 
    $$
        \left(\frac{u^3}{v^2}\right)(P) = 1.
    $$
    After dividing the relation $u^q + \gamma u + G(v) =0$ by $v^{2q/3}$ and evaluating the resulting functions at $P$ we get that
    $$
        \left(\frac{G(v)}{v^{2q/3}}\right)(P) = - \left(\frac{u^q}{v^{2q/3}} + \frac{\gamma u}{v^{2q/3}}\right)(P) = - \left( \left(\frac{u^3}{v^2}\right)(P)\right)^{q/3}= -1,
    $$
    since $v_P(\gamma u/v^{2q/3}) >0$. From this we conclude that the leading coefficient of $G(v)$ is indeed $-1$.

    The claimed identity follows almost directly from the proof of the previous proposition: Keeping the same notation as in the proof we have $S(T) = G'(v)$ and $\alpha u^q - \eta^q\delta u = -\alpha G(v)$ so Equation \eqref{eq:differentiated_hermitian} can be rewritten as
    $$
        G'(v)\left(\mu^q - \alpha G(v)\right) = \gamma(\beta v^q + \eta^q v + \nu^q).
    $$
    It follows immediately that $\deg(G'(v))=q/3$, since $\deg G = 2q/3$, and the leading coefficient can be read off directly; it is $\ell := \gamma\beta/\alpha$ as claimed. Moreover, after differentiating with respect to $v$ and rearranging we get
    $$
        G''(v)(\mu^q-\alpha G(v)) = \alpha (G'(v))^2 + \gamma \eta^q,
    $$
    from which $G''(v) = \ell^2 \in \Fqq$ follows. Solving for $G(v)$ gives  
    $$
        G(v) = -\left(\frac{G'(v)}{\ell}\right)^2 + \theta,
    $$
    where $\theta := \frac{\mu^q}{\alpha} - \frac{\gamma \eta^q}{\alpha \ell^2} \in \Fqq$, as claimed.
\end{proof}

Finally, we finish the proof of the main theorem by exhibiting a change of variables that turns $u^q + \gamma u + G(v)=0$ into the equation for the known third largest genus function field. \newline

\begin{proof}[Proof of Theorem~\ref{thm:main}]
    Recall from Corollary \ref{cor:cubic_hermitian_shape} that the Hermitian surface relation reads
    $$
        \rho + \alpha u^{q+1} + \beta v^{q+1} + \mu u^q + \mu^q u + \nu v^q + \nu^qv + \eta w^q + \eta^qw = 0. 
    $$
    By (repeatedly) substituting 
    $$
        w = u^3 + \delta u^2 -v^2 \quad \text{and} \quad u^q = -\gamma u - G(v)
    $$
    we can obtain a polynomial identity in $u$ and $v$ where no term has $u$-degree above $3$. In fact, we get 
    \begin{equation}\label{eq:hermitian_cubic_in_u}
    \begin{aligned}
        & (-\eta \gamma^3 + \eta^q) u^3 + (-\alpha\gamma + \eta \delta^q \gamma^2 + \eta^q\delta) u^2 + \left(\mu^q - \mu \gamma - G(v)(\alpha+\eta \delta^q \gamma)\right)u  \\
        & + \left( -\eta v^{2q} +\beta v^{q+1} + \nu v^q - \eta^q v^2 + \nu^q v - \eta G(v)^3+ \eta \delta^q G(v)^2 - \mu G(v) + \rho \right) = 0
    \end{aligned}
    \end{equation}
    Since $[F:\Fqq(v)] = q > 3$ the left hand side must be zero also when considered as a polynomial in $u$. From considering the coefficient in front of $u^3$ we get
    $$
        \gamma^3 = \eta^{q-1}.
    $$
    Since $G(v)$ is a non-constant polynomial in $v$ the coefficient in front of $u$ tells us that
    $$
        \mu^q = \mu \gamma \quad \text{and} \quad \alpha = -\eta \delta^q\gamma.
    $$
    Combining the above with $\gamma =-\frac{\delta \eta^q}{\alpha}$ from Proposition \ref{prop:shape_of_phi} we obtain
    $$
\gamma^2=\frac{\eta^{q-1}}{\delta^{q-1}} \quad \text{and hence}  \quad      \gamma = \frac{\gamma^3}{\gamma^2}=\delta^{q-1}. 
    $$
This implies
    $$
        \frac{\mu}{\delta} = \frac{\mu^q}{\delta^q}.
    $$
    so that $\frac{\mu}{\delta} \in \Fq$. The last fact will be useful later, but for now we note that $\gamma = \delta^{q-1}$ implies
    $$
        \left(\frac{u}{\delta}\right)^q + \left(\frac{u}{\delta}\right) = \frac{u^q + \gamma u}{\delta^q} = \frac{-G(v)}{\delta^q},
    $$
    which already brings us closer to the desired form. Inserting the expression for $G(v)$ from the previous lemma we get
    $$
        \left(\frac{u}{\delta}\right)^q + \left(\frac{u}{\delta}\right) = \frac{1}{\delta^q}\cdot\left(\frac{G'(v)}{\ell}\right)^2 - \frac{\theta}{\delta^q}.
    $$
    The constant $\theta/\delta^q$ is in $\Fq$ and can be absorbed in the left hand side. Indeed, 
    \begin{align*}
        \frac{\theta}{\delta^q} = \frac{1}{\delta^q} \left( \frac{\mu^q}{\alpha} + \frac{\alpha^2}{\delta\beta^2}\right) = \frac{1}{\alpha}\cdot\left(\frac{\mu}{\delta}\right)^q + \frac{1}{\delta^{q+1}}\cdot \frac{\alpha^2}{\beta^2} \in \Fq,
    \end{align*}
    so 
    $$
        u_1 := \frac{u}{\delta} - \frac{\theta}{\delta^q}
    $$
    satisfies
    $$
        u_1^q + u_1= \left(\frac{u}{\delta}\right)^q + \left(\frac{u}{\delta}\right) - 2\frac{\theta}{\delta^q} =  \frac{1}{\delta^q}\cdot\left(\frac{G'(v)}{\ell}\right)^2.
    $$

    Next, we show that $\delta$, and hence $1/\delta^q$, is a square in $\Fqq$ . To do so, recall that the constant term in Equation \eqref{eq:hermitian_cubic_in_u} must be zero, i.e.,
    $$
        -\eta v^{2q} +\beta v^{q+1} + \nu v^q - \eta^q v^2 + \nu^q v - \eta G(v)^3+ \eta \delta^q G(v)^2 - \mu G(v) + \rho = 0.
    $$
    Note that this also holds as an equality of polynomials in $\Fqq[v]$. The same is true for the identities from Lemma \ref{lemma:G_v}, and working modulo $G'(v)$ in $\Fqq[v]$ these imply
    $$
        v^q \equiv -\frac{\eta^qv + \nu^q}{\beta} \pmod{G'(v)} \quad \text{and} \quad G(v) \equiv \theta \pmod{G'(v)},
    $$
    which when combined with the above yields
    $$
        \left( \eta^q - \frac{\eta^{2q+1}}{\beta^2}\right) v^2 + \left( \frac{\eta^{q+1}\nu^q}{\beta^2} - \frac{\nu\eta^q}{\beta}\right)v + \tau \equiv 0 \pmod{G'(v)},
    $$
    for some $\tau \in \Fqq$. Since $\deg(G'(v)) = q/3 \geq 3$ the left hand side must be zero, which implies 
    \begin{equation}\label{eq:beta_eta_nu}
        \beta^2 = \eta^{q+1} \quad \text{and} \quad \eta\nu^q = \beta\nu.
    \end{equation}
    From earlier, we have $\alpha = -\eta \delta^q\gamma = -\eta \delta^q\delta^{q-1}$, and raising to the $(q+1)$'st power gives 
    $$
        \alpha^2 = \alpha^{q+1} = \eta^{q+1}\delta^{q^2+q}= \beta^2\delta^{q+1}, 
    $$
    using $\beta^2 = \eta^{q+1}$ as we just saw. This means that
    $$
        \delta^{\frac{q^2-1}{2}} = \left(\frac{\alpha}{\beta}\right)^{q-1} = 1,
    $$
    so $\delta$ is a square. With $\lambda\in \Fqq$ chosen such that $\lambda^2 = \delta^q$ we then have
    $$
        u_1^q + u_1 = \left( \frac{G'(v)}{\lambda \ell}\right)^2.
    $$

    Our goal now is to find an affine $\Fqq$-transformation of $v$, say $v_1 = av+b$, $a,b\in \Fqq$, such that 
    $$
        \frac{G'(v)}{\lambda \ell} = v_1 + v_1^3 + \dots + v_1^{\frac{q}{3}}.
    $$
    Since $G'' = \ell^2$ (see the proof of the previous lemma), this would imply 
    $$
        \frac{\ell}{\lambda} dv = \frac{G''}{\ell\lambda}dv = d\left(\frac{G'}{\ell \lambda}\right) = dv_1,
    $$
    so we take $a := \ell/\lambda$. To find a good candidate for $b$, note that the above would imply
    \begin{equation}\label{eq:finding_b}
        G'\left(\frac{-b}{a}\right) = G'\left(\frac{-b\lambda}{\ell}\right)= 0,
    \end{equation}
    since the right hand side of the equation we want is $0$ at $v_1=0$. We claim that there exists such $b$ in $\Fqq$: Let $c_0$ be a root of the polynomial $G'(v)$ in some algebraic closure of $\Fqq$. Then, Lemma \ref{lemma:G_v} implies
    $$
        \beta c_0^q + \eta^q c_0 + \nu^q = 0.
    $$
    Using this, we find
    $$
        c_0^{q^2} = \left(\frac{-\eta^q c_0 - \nu^q}{\beta}\right)^q = -\frac{\eta}{\beta}\left( \frac{-\eta^q c_0 - \nu^q}{\beta}\right) - \frac{\nu}{\beta} = \frac{\eta^{q+1}}{\beta^2}c_0 + \frac{\eta \nu^q- \beta\nu}{\beta^2} = c_0,
    $$
    where the last equality follows from Equation \eqref{eq:beta_eta_nu}. We conclude that $c_0 \in \Fqq$, so we choose $b := -\ell c_0/\lambda$. 

    Finally, we need to check that our choices for $a$ and $b$ actually give the relation we want, i.e., that $v_1 = \frac{\ell}{\lambda}(v-c_0)$, with $c_0\in \Fqq$ a root of $G'$ and $\lambda \in \Fqq$ such that $\lambda^2 = \delta^q$, satisfies
    $$
        \frac{G'(v)}{\lambda \ell} = v_1 + v_1^3 + \dots + v_1^{\frac{q}{3}}.
    $$
    We claim that
    \begin{equation}\label{eq:congruence_mod_v_1q}
        \left( \frac{G'(v)}{\lambda \ell}\right)^3 - \frac{G'(v)}{\lambda \ell} \equiv -v_1 \pmod{v_1^q},
    \end{equation}
    when viewing $G'(v)$ as a polynomial in $v_1$. To see this, we return to the two equations from Lemma \ref{lemma:G_v}. Combining them yields
    $$
        G'(v) \left(\mu^q + \frac{\alpha}{\ell^2}(G'(v))^2 - \alpha \theta \right) = \gamma (\beta v^q + \eta^q v + \nu^q),
    $$
    which can be rewritten as 
    $$
        \frac{\alpha}{\ell^2}(G'(v))^3 + (\mu^q - \alpha \theta)G'(v) = \gamma(\beta v^q + \eta^q v + \nu^q).
    $$
    We saw previously that $\alpha^2 = \beta^2\delta^{q+1}$, and from this it follows that 
    $$
        \theta = \frac{\mu^q}{\alpha} + \frac{\alpha^2}{\delta\beta^2} = \frac{\mu^q}{\alpha}+\delta^q.
    $$
    Using this and $\gamma\beta = \alpha \ell$, we can rewrite the equation above as
    $$
        \left(\frac{G'(v)}{\ell}\right)^3 - \frac{\delta^q}{\ell}G'(v) = v^q + \frac{\eta^q}{\beta}v + \frac{\nu^q}{\beta},
    $$
    which after dividing by $\lambda^3$ and using $\lambda^2 = \delta^q$ yields
    $$
        \left(\frac{G'(v)}{\lambda\ell}\right)^3 - \frac{G'(v)}{\lambda\ell} = \frac{1}{\lambda^3}\left( v^q + \frac{\eta^q}{\beta}v + \frac{\nu^q}{\beta} \right).
    $$
    Now, substituting $v = \frac{\lambda}{\ell}v_1 + c_0$ on the right, and working modulo $v_1^q$ gives
    $$
        \left(\frac{G'(v)}{\lambda\ell}\right)^3 - \frac{G'(v)}{\lambda\ell} \equiv \frac{1}{\lambda^3}\left( c_0^q + \frac{\eta^q}{\beta}\left(\frac{\lambda}{\ell}v_1 + c_0\right) + \frac{\nu^q}{\beta} \right) \pmod{v_1^q} 
    $$
    and the right hand side is simply $-v_1$ since $\beta c_0^q + \eta^q c_0 + \nu^q = 0$ per definition, and 
    $$
        \frac{1}{\lambda^3}\cdot \frac{\eta^q\lambda}{\beta\ell} = \frac{\eta^q \alpha}{\gamma\beta^2\delta^q} = \frac{-\eta^{q+1}}{\beta^2} = -1,
    $$
    using first $\ell = \gamma\beta/\alpha$, then $\alpha = - \eta \delta^q\gamma$ and finally $\beta^2 = \eta^{q+1}$. 
    
    We conclude that (\ref{eq:congruence_mod_v_1q}) holds, and using it repeatedly implies
    $$
        \frac{G'(v)}{\lambda \ell} \equiv v_1 + v_1^3 + \dots + v_1^{\frac{q}{3}} + \left( \frac{G'(v)}{\lambda \ell}\right)^q\pmod{v_1^q},
    $$
    as polynomials in $v_1$. Moreover, the polynomial $G'(v)$ has no constant term when viewed as a polynomial in $v_1$, since $v \equiv c_0 \pmod{v_1}$ and $c_0$ is a root of $G'$. This means that $(G'(v))^q$ is divisible by $v_1^q$, so we have
    $$
        \frac{G'(v)}{\lambda \ell} \equiv v_1 + v_1^3 + \dots + v_1^{\frac{q}{3}} \pmod{v_1^q}.
    $$
    The $v_1$-degree of both the left hand side and the right hand side is strictly less than $q$, so the above is an equality, i.e.
    $$
        \frac{G'(v)}{\lambda \ell} = v_1 + v_1^3 + \dots + v_1^{\frac{q}{3}},
    $$
    which was what we wanted to show. 

    Writing $y=u_1$ and $x=v_1$, we conclude that 
    $$
        y^q + y = \left(x + x^3 + \dots + x^{\frac{q}{3}}\right)^2,
    $$
    and $F = \Fqq(u,v) = \Fqq(u_1,v_1)=\Fqq(x,y)$, so this finishes the proof of the main theorem.
\end{proof}

\section{Acknowledgments and AI tool disclosure}
This work was supported by a research grant (VIL''52303'') from Villum Fonden. This work was supported in part by a grant of access to OpenAI models through the ChatGPT for Academic Researchers program.

GPT 6 Astra was used to obtain a first version of the proof of the main result. Afterwards the authors used GPT 6 Astra and GPT 5.6 Sol to improve and simplify different parts of the proofs.
We thoroughly checked the correctness of all results in this article. 
All text, including the wording of the proofs, was written by the authors.

\bibliographystyle{abbrv}
\bibliography{ref}

\end{document}